\documentclass[reqno]{amsart}

\usepackage{amsmath,amssymb,amsthm,amsfonts,enumerate, enumitem,hyperref,cleveref,multicol}
\usepackage{dsfont}
\usepackage[all]{xy}
\usepackage[T1]{fontenc}
\usepackage{tikz}
\usepackage{pgfplots}
\pgfplotsset{compat=1.15}
\usepackage{mathrsfs}
\usetikzlibrary{arrows}

\usepackage[margin=1in]{geometry}

\DeclareMathOperator{\Aut}{Aut}

\DeclareMathOperator{\spn}{span}
\DeclareMathOperator{\Ann}{Ann}

\theoremstyle{plain}

\newtheorem{thrm}{Theorem}[section]
\newtheorem{cor}[thrm]{Corollary}
\newtheorem{prop}[thrm]{Proposition}
\newtheorem{lem}[thrm]{Lemma}

\theoremstyle{definition}

\newtheorem{rem}[thrm]{Remark}

\crefrangeformat{equation}{#3(#1)#4--#5(#2)#6}

\crefname{thrm}{Theorem}{Theorems}
\crefname{theorem}{Theorem}{Theorems}
\crefname{lem}{Lemma}{Lemmas}
\crefname{cor}{Corollary}{Corollaries}
\crefname{prop}{Proposition}{Propositions}
\crefname{defn}{Definition}{Definitions}
\crefname{exm}{Example}{Examples}
\crefname{rem}{Remark}{Remarks}
\crefname{conj}{Conjecture}{Conjectures}
\crefname{quest}{Question}{Questions}
\crefname{section}{Section}{Sections}
\crefname{equation}{\unskip}{\unskip}
\crefname{enumi}{\unskip}{\unskip}
\crefname{subsection}{Subsection}{Subsections}

\renewcommand{\iff}{\Leftrightarrow}
\newcommand{\impl}{\Rightarrow}

\newcommand{\af}{\alpha}
\newcommand{\bt}{\beta}
\newcommand{\gm}{\gamma}
\newcommand{\dl}{\delta}

\newcommand{\lb}{\lambda}
\newcommand{\sg}{\sigma}
\newcommand{\vf}{\varphi}

\newcommand{\A}{\mathcal{A}}

\newcommand{\M}{\mathbf{M}}
\newcommand{\I}{\mathbf{I}}
\newcommand{\T}{\mathbf{T}}

\newcommand{\sst}{\subseteq}

\newcommand{\id}{\mathrm{id}}

\newcommand{\m}{{}^{-1}}

\begin{document}
	\title[$\sg$-matching and interchangeable structures on $UT_n(K)$]{$\sg$-matching and interchangeable structures\\ on the strictly upper triangular matrix algebra}	

\author{Mykola Khrypchenko}
\address{Departamento de Matem\'atica, Universidade Federal de Santa Catarina,  Campus Reitor Jo\~ao David Ferreira Lima, Florian\'opolis, SC, CEP: 88040--900, Brazil}
\email[Corresponding author]{nskhripchenko@gmail.com}


\begin{abstract}
	We describe $\sg$-matching, interchangeable and, as a consequence, totally compatible structures on the strictly upper triangular matrix algebra $UT_n(K)$ for all $n\ge 3$.
\end{abstract}

\maketitle

%\tableofcontents

	%\newpage
	\section*{Introduction}	
	
	Given two structures from a same variety of algebras $\A$ (e.g. associative, Lie, Leibniz algebras) defined on a same vector space, in general their sum does not necessarily belong to $\A$. If it does, such structures are called \textit{compatible}.
	The notion of compatibility goes back to old papers in mathematical physics~\cite{Magri78,Reyman-Semenov89,Bolsinov91}, in which context it has been used for quite a while~\cite{Carinena-Grabowski-Marmo2000,Bolsinov-Borisov02,Golubchik-Sokolov02,Golubchik-Sokolov05,Odesskii-Sokolov06}. More precisely,~\cite{Carinena-Grabowski-Marmo2000,Odesskii-Sokolov06} dealt with compatible associative algebras, while~\cite{Golubchik-Sokolov02,Golubchik-Sokolov05} studied compatible Lie algebras. Since then, the interest in compatible algebras has gone beyond mathematical physics: thus, there appeared compatible Leibniz algebras~\cite{Makhlouf-Saha23}, pre-Lie algebras~\cite{Abdelwahab-Kaygorodov-Makhlouf24}, anti-pre-Lie algebras~\cite{Normatov24}, Hom-Lie algebras~\cite{Das23CompHomLie} and Hom-Lie triple systems~\cite{Teng-Long-Zhang-Lin23} and so on. For the algebraic classification of nilpotent compatible algebras of small dimensions see~\cite{Abdelwahab-Kaygorodov-Makhlouf24,Ladra-Cunha-Lopes24}.
	  
	The compatibility of a pair of bilinear associative products $\cdot_1$ and $\cdot_2$ on a same vector space $V$ is equivalent to the following identity on $V$:
	 \begin{align}\label{compatible-ass-identity}
	 	(a\cdot_1 b)\cdot_2 c + (a\cdot_2 b)\cdot_1 c = a\cdot_1(b\cdot_2 c) + a\cdot_2(b\cdot_1 c).
	 \end{align}
	 Such structures first appeared (under the name ``quantum bi-Hamiltonian systems'') in~\cite{Carinena-Grabowski-Marmo2000}. Given an associative algebra $(A,\cdot)$, it is natural to ask: what are the associative products on $A$ that are compatible with $\cdot$? In general, one can point out some specific classes of such products (see, for example,~\cite{Carinena-Grabowski-Marmo2000}), but their full description seems to be a difficult task. For instance, when $A$ is the matrix algebra, such a description is given in~\cite{Odesskii-Sokolov06}. 
	 
	 However, there are several particular cases of \cref{compatible-ass-identity} based on equalities of pairs monomials. If the $i$-th monomial of the left hand side of \cref{compatible-ass-identity} equals the $i$-th monomial of the right hand side of \cref{compatible-ass-identity} for all $i\in\{1,2\}$, then the products $\cdot_1$ and $\cdot_2$ are said to be \textit{matching} (in this case $(V,\cdot_1,\cdot_2)$ is a so-called \textit{matching dialgebra}~\cite{zbg12}, also known as \textit{$As^{(2)}$-algebra}~\cite{Zinbiel}). In~\cite{zgg23}, this notion was generalized to \textit{$\sg$-matching} products, where $\sg$ is a permutation of $\{1,2\}$ (so that the $i$-th monomial of the left hand side of \cref{compatible-ass-identity} equals the $\sg(i)$-th monomial of the right hand side of \cref{compatible-ass-identity}). If all the $4$ monomials in \cref{compatible-ass-identity} are equal, then $\cdot_1$ and $\cdot_2$ are \textit{totally compatible} (see~\cite{zbg13,Zhang13}). In~\cite{Khr2024}, we also studied one more ``compatibility-like'' condition, called \textit{interchangeability}, whose idea comes from~\cite{Strohmayer08} and which means that one can permute the operations $\cdot_1$ and $\cdot_2$ in each of the monomials in \cref{compatible-ass-identity} without changing the brackets. It turns out that, whenever $A$ is unital, $\sg$-matching, interchangeable and totally compatible structures on $A$ admit easy descriptions~\cite{Khr2024}, so it makes sense to study them on algebras that are far from being unital, for example, on nilpotent ones.
	
	In this paper we characterize $\sg$-matching, interchangeable and totally compatible structures on the strictly upper triangular matrix algebra $UT_n(K)$, $n\ge 3$, which is a classical example of a nilpotent associative algebra. In \cref{sec-prelim} we recall the definitions of different types of compatibility, give examples and discuss isomorphic and antiisomorphic structures on an arbitrary associative algebra $A$. We also give some technical details on $UT_n(K)$ that will be used below. \cref{sec-id-match,sec-(12)-match,sec-interchangeable} are devoted to $\id$-matching, $(12)$-matching and interchangeable structures on $UT_n(K)$, respectively. Each of these sections has a theorem that describes the corresponding structure as a linear combination of some specific products on $UT_n(K)$ (these are \cref{id-matching-UT_n-full-descr,(12)-matching-UT_n-full-descr,interchangeable-UT_n-descr}, respectively) and a subsection, where we apply the theorem to give a complete list of pairwise non-isomorphic structures on $UT_3(K)$ (see \cref{id-matching-UT_3-full-descr,(12)-matching-UT_3-full-descr,interchangeable-UT_3-full-descr}). As a consequence, in \cref{sec-totally-comp} we obtain a description of totally compatible structures on $UT_n(K)$.
	%	%\newpage	
	\section{Definitions and preliminaries}\label{sec-prelim}
	
	In what follows $K$ will be an arbitrary field. All the algebras and vector spaces will be over $K$, all the maps between them will be $K$-linear and all the products on them will be $K$-bilinear. 
	
	\subsection{Matching compatibilities and interchangeability}
	
		Recall from \cite[Definition 1.1]{Khr2024} (based on~\cite{zgg23}) that 
		two associative products $\cdot_1$ and $\cdot_2$ on a same vector space $V$ are 
			\begin{enumerate}
				\item \textit{$\sg$-matching} (where $\sg\in S_2=\{\id,(12)\}$), if
				\begin{align}\label{matching-ass-identity}
					(a\cdot_1 b)\cdot_2c = a\cdot_{\sg(1)}(b\cdot_{\sg(2)}c) \text{ and } (a\cdot_2b)\cdot_1
					 c = a\cdot_{\sg(2)}(b\cdot_{\sg(1)}c),
				\end{align}
				\item \textit{totally compatible}, if 
				\begin{align}\label{totally-comp-ass-identity}
					(a\cdot_1 b)\cdot_2 c = (a\cdot_2 b)\cdot_1 c = a\cdot_1 (b\cdot_2 c) = a\cdot_2 (b\cdot_1 c),
				\end{align}
			\end{enumerate}
			for all $a,b,c\in V$. 
			
			 Furthermore, two (not necessarily associative) products $\cdot_1$ and $\cdot_2$ on $V$ are said to be \textit{interchangeable} if
		\begin{align*}%\label{interchangeable-ass-identity}
			(a\cdot_1 b)\cdot_2c = (a\cdot_2 b)\cdot_1 c \text{ and } a\cdot_1(b\cdot_2 c) = a\cdot_2 (b\cdot_1 c).
		\end{align*}
	It is clear that $\cdot_1$ and $\cdot_2$ are totally compatible $\iff$ $\cdot_1$ and $\cdot_2$ are $\sg$-matching for all $\sg\in S_2$ $\iff$ $\cdot_1$ and $\cdot_2$ are interchangeable and $\sg$-matching for some $\sg\in S_2$.

Given an associative algebra $(A,\cdot)$, by a \textit{$\sg$-matching} (resp. \textit{interchangeable} or \textit{totally compatible}) \textit{structure} on $A$ we mean an associative product $*$ on $A$ that is $\sg$-matching (resp. interchangeable or totally compatible) with $\cdot$. It is obvious that $\cdot$ itself is totally compatible with $\cdot$. More generally, the \textit{mutation}~\cite{ElduqueMyung} $\cdot_x$ of $\cdot$ by an arbitrary $x\in A$ (i.e., $a\cdot_x b=a\cdot x\cdot b$) is $\id$-matching with $\cdot$. If, moreover, $x$ is central (more generally, belongs to the \textit{centroid} of $A$), then $\cdot_x$ is totally compatible with $\cdot$. For \textit{unital} algebras $A$, all the $\id$-matching structures on $A$ are mutations of $\cdot$ (see \cite[Proposition 2.2]{Khr2024}), and all the $(12)$-matching or interchangeable structures on $A$ are mutations of $\cdot$ by central elements of $A$ (see \cite[Proposition 2.6]{Khr2024}). In particular, any $(12)$-matching or interchangeable structure on a unital algebra $(A,\cdot)$ is always totally compatible with $\cdot$. If $A$ is non-unital (for example, if $A$ is \textit{nilpotent}), all these matching compatibilities may be different as it was shown in~\cite[Examples 1.4 and 1.5]{Khr2024}. 

In order to give another general example of totally compatible structures (which makes sense, in particular, for nilpotent algebras), recall some notations. Let $(A,\cdot)$ be a (not necessarily associative) algebra. The \textit{annihilator} of $A$ is the ideal $\Ann(A)=\{a\in A: a\cdot b=b\cdot a=0\text{ for all }b\in A\}$. As usual, we denote by $A^2$ the subspace $\spn_K\{a\cdot b: a,b\in A\}$.
\begin{rem}
	Let $(A,\cdot)$ be an algebra and $V$ be a vector space complement of $A^2$. Then any bilinear map $\mu:V\times V\to\Ann(A)$ extends to the product $*$ on $A$ given by
	\begin{align}\label{annihilator-valued-structure}
		(v_1+a_1)*(v_2+a_2)=\mu(v_1,v_2),
	\end{align}
	where $v_i\in V$, $a_i\in A^2$, $i=1,2$. Moreover, $(a\cdot b)* c = (a*b)\cdot c = a\cdot (b*c) = a*(b\cdot c)=0$. In particular, $\cdot$ and $*$ are totally compatible, whenever they are associative.
\end{rem}

\subsection{Isomorphic and antiisomorphic structures}
Let $*_1$ and $*_2$ be two $\sg$-matching, interchangeable or totally compatible structures  on $(A,\cdot)$. We say that $*_1$ and $*_2$ are \textit{isomorphic}, if there exists $\phi\in\Aut(A,\cdot)$ such that 
\begin{align*}
	\phi(a*_1b)=\phi(a)*_2\phi(b)
\end{align*}
for all $a,b\in A$ (this means that $(A,\cdot,*_1)$ and $(A,\cdot,*_2)$ are isomorphic as algebras with two multiplications). In particular, given a $\sg$-matching (resp. interchangeable or totally compatible) structure $*$ on $(A,\cdot)$ and $\phi\in\Aut(A,\cdot)$, one can define the new product $\star$ on $A$ by means of
\begin{align}\label{a-star-b=phi(phi-inv(a)*phi-inv(b))}
	a\star b=\phi(\phi\m(a)*\phi\m(b))
\end{align}
for all $a,b\in A$. Clearly, $\star$ is also a $\sg$-matching (resp. interchangeable or totally compatible) structure on $(A,\cdot)$ that is isomorphic to $*$.

An antiautomorphism of $(A,\cdot)$ also induces a new structure on $A$ of the same matching type, which is, however, not isomorphic to $*$.\footnote{More precisely, the resulting algebras $(A,\cdot,*_1)$ and $(A,\cdot,*_2)$ are antiisomorphic.}
\begin{lem}\label{antiauto=>structure-of-the-same-type}
	Let $(A,\cdot)$ be an algebra and $\vf$ be an antiautomorphism of $(A,\cdot)$. Given a product $*$ on $A$, let $\star$ be the new product on $A$ given by 
	\begin{align}\label{a-star-b=phi(phi-inv(b)*phi-inv(a))}
		a\star b=\phi(\phi\m(b)*\phi\m(a))
	\end{align}
	for all $a,b\in A$.
	\begin{enumerate}
		\item\label{star-sg-matching} If $*$ satisfies \cref{matching-ass-identity} with some $\sg\in S_2$, then so does $\star$.
		\item\label{star-interchangeable} If $*$ is interchangeable with $\cdot$, then so is $\star$.
		\item\label{star-associative} If $*$ is associative, then so is $\star$.
	\end{enumerate}
\end{lem}
\begin{proof}
	Let $a,b,c\in A$. We first calculate
	\begin{align*}
		(a\cdot b)\star c&=\phi(\phi\m(c)*\phi\m(a\cdot b))=\phi(\phi\m(c)*(\phi\m(b)\cdot\phi\m(a))),\\
		(a\star b)\cdot c&=\phi(\phi\m(b)*\phi\m(a))\cdot \phi(\phi\m(c))=\phi(\phi\m(c)\cdot(\phi\m(b)*\phi\m(a))).
	\end{align*}
	
	\textit{\cref{star-sg-matching}}. Assume that $*$ satisfies \cref{matching-ass-identity} with $\sg=\id$. Then
	\begin{align*}
		\phi(\phi\m(c)*(\phi\m(b)\cdot\phi\m(a)))
		&=\phi((\phi\m(c)*\phi\m(b))\cdot\phi\m(a))
		=\phi(\phi\m(a))\cdot\phi(\phi\m(c)*\phi\m(b))=a\cdot(b\star c),\\
		\phi(\phi\m(c)\cdot(\phi\m(b)*\phi\m(a)))
		&=\phi((\phi\m(c)\cdot\phi\m(b))*\phi\m(a))
		=\phi(\phi\m(b\cdot c)*\phi\m(a))=a\star(b\cdot c).
	\end{align*}
	Similarly, if $*$ satisfies \cref{matching-ass-identity} with $\sg=(12)$, then
	\begin{align*}
		\phi(\phi\m(c)*(\phi\m(b)\cdot\phi\m(a)))
		&=\phi((\phi\m(c)\cdot\phi\m(b))*\phi\m(a))
		=\phi(\phi\m(b\cdot c)*\phi\m(a))=a\star(b\cdot c),\\
		\phi(\phi\m(c)\cdot(\phi\m(b)*\phi\m(a)))
		&=\phi((\phi\m(c)*\phi\m(b))\cdot\phi\m(a))
		=\phi(\phi\m(a))\cdot\phi(\phi\m(c)*\phi\m(b))=a\cdot(b\star c).
	\end{align*}
	
	\textit{\cref{star-interchangeable}}. Assume that $*$ is interchangeable with $\cdot$. Then
	\begin{align*}
		\phi(\phi\m(c)*(\phi\m(b)\cdot\phi\m(a)))
		&=\phi(\phi\m(c)\cdot(\phi\m(b)*\phi\m(a)))
		=\phi(\phi\m(b)*\phi\m(a))\cdot\phi(\phi\m(c))=(a\star b)\cdot c,\\
		\phi(\phi\m(c)\cdot(\phi\m(b)*\phi\m(a)))
		&=\phi(\phi\m(c)*(\phi\m(b)\cdot\phi\m(a)))
		=\phi(\phi\m(c)*\phi\m(a\cdot b))=(a\cdot b)\star c.
	\end{align*}
	
	\textit{\cref{star-associative}}. Assume that $*$ is associative. Then
	\begin{align*}
		(a\star b)\star c&=\phi(\phi\m(c)*\phi\m(a\star b))=\phi(\phi\m(c)*(\phi\m(b)*\phi\m(a)))\\
		&=\phi((\phi\m(c)*\phi\m(b))*\phi\m(a))=\phi(\phi\m(b\star c)*\phi\m(a))=a\star(b\star c).\qedhere
	\end{align*}
\end{proof}

\subsection{The algebras $T_n(K)$ and $UT_n(K)$}

Let $(T_n(K),\cdot)$ be the algebra of $n\times n$ upper triangular matrices over $K$. As a $K$-space, $T_n(K)$ admits the canonical basis consisting of the matrix units $e_{ij}$ (denoted also $e_{i,j}$ when needed), $1\le i\le j\le n$, where $e_{ij}\cdot e_{kl}=\dl_{jk}e_{il}$. Given $a\in T_n(K)$, we write $a=\sum_{i\le j}a(i,j)e_{ij}$ and call $a(i,j)$ the \textit{value} of $a$ at $(i,j)$. The following straightforward formulas will be used throughout the text without any reference: 
\begin{align*}
	e_{kl}\cdot a=\sum_{j=l}^{n}a(l,j)e_{kj}\text{ and }a\cdot e_{kl}=\sum_{i=1}^{k}a(i,k)e_{il}.
\end{align*}
When needed, we will write $e_i$ for $e_{ii}$, $1\le i\le n$.

Denote by $UT_n(K)$ the ideal of $T_n(K)$ consisting of strictly upper triangular matrices. Thus, $UT_n(K)$ is a nilpotent associative algebra. Clearly, 
\begin{align*}
	UT_n(K)=\spn_K\{e_{ij}: 1\le i<j\le n\}.
\end{align*}
As a $K$-algebra, $UT_n(K)$ is generated by $\{e_{i,i+1}:  1\le i<n\}$.

\begin{rem}
	We have $\Ann(UT_n(K))=\spn_K\{e_{1n}\}$ and $UT_n(K)^2=\spn_K\{e_{ij}: 2\le i+1<j\le n\}$.
\end{rem}

 The following obvious involution of $UT_n(K)$ in combination with \cref{antiauto=>structure-of-the-same-type} will permit us to reduce a lot of calculations in \cref{sec-(12)-match,sec-interchangeable}.
 
\begin{rem}\label{inv-phi}
	The linear map $\phi:UT_n(K)\to UT_n(K)$ given for all $1\le i<j\le n$ by 
	\begin{align}\label{vf(e_kl)=e_(n-l+1_n-k+1)}
		\phi(e_{ij})=e_{n-j+1,n-i+1}
	\end{align}
	is an involution of $UT_n(K)$.
\end{rem}

\begin{cor}\label{from-*-to-star}
	Let $\phi$ be the involution \cref{vf(e_kl)=e_(n-l+1_n-k+1)} of $UT_n(K)$. Given a product $*$ on $UT_n(K)$, let $\star$ be the new product on $A$ given by \cref{a-star-b=phi(phi-inv(b)*phi-inv(a))}.
	\begin{enumerate}
		\item If $*$ satisfies \cref{matching-ass-identity} with some $\sg\in S_2$, then so does $\star$.
		\item If $*$ is interchangeable with $\cdot$, then so is $\star$.
		\item If $*$ is associative, then so is $\star$.
	\end{enumerate}
\end{cor}

The case $n=2$ is trivial, since $UT_2(K)$ is a one-dimensional algebra with zero multiplication, so any bilinear product on $UT_2(K)$ is totally compatible with $\cdot$. Thus, we assume $n\ge 3$ below.

The description of automorphisms of $UT_3(K)$ will be used several times in the applications, so let us give it here.

\begin{lem}
	The automorphisms of $UT_3(K)$ are the linear maps $\phi:UT_3(K)\to UT_3(K)$ of the form
	\begin{align}\label{aut-UT_3-on-e_12_e_23_e_13}
		\phi(e_{12})=a_{11}e_{12}+a_{31}e_{13},\ \phi(e_{23})=a_{22}e_{23}+a_{32}e_{13},\ \phi(e_{13})=a_{11}a_{22}e_{13},
	\end{align}
	where $a_{11},a_{22}\in K^*$ and $a_{31},a_{32}\in K$.
\end{lem}
\begin{proof}
	It is easy to check that any linear map of the form \cref{aut-UT_3-on-e_12_e_23_e_13} is an automorphism of $UT_3(K)$. Conversely, let $\phi\in\Aut(UT_3(K))$. Clearly, $\phi(e_{13})=a_{33}e_{13}$ for some $a_{33}\in K^*$, because $\spn_K\{e_{13}\}=\Ann(UT_3(K))$. Write $\phi(e_{12})=a_{11}e_{12}+a_{21}e_{23}+a_{31}e_{13}$ and $\phi(e_{23})=a_{12}e_{12}+a_{22}e_{23}+a_{32}e_{13}$. Then $\phi(e_{12})\phi(e_{23})=\phi(e_{13})$ gives $a_{11}a_{22}=a_{33}$. In particular, $a_{11},a_{22}\in K^*$. Moreover, $\phi(e_{12})^2=0\iff a_{11}a_{21}=0\iff a_{21}=0$ and similarly $\phi(e_{23})^2=0\iff a_{12}a_{22}=0\iff a_{12}=0$.
\end{proof}
	\section{$\id$-matching structures on $UT_n(K)$}\label{sec-id-match}

We begin this section with the following general fact.
\begin{lem}\label{*_1=*_2-on-generators}
	Let $(A,\cdot)$ be an associative algebra and $X$ a set of generators of $A$. Any bilinear product on $A$ satisfying \cref{matching-ass-identity} with $\sg=\id$ is uniquely determined by its values on $X\times X$.
\end{lem}
\begin{proof}
	Let $*_1$ and $*_2$ be two bilinear products on $A$ satisfying \cref{matching-ass-identity} with $\sg=\id$ and such that $*_1|_{X\times X}=*_2|_{X\times X}$. Then for all $x_1,x_2,y_1,y_2\in X$ we have 
	\begin{align*}
		(x_1\cdot x_2)*_1(y_1\cdot y_2)=x_1\cdot (x_2*_1y_1)\cdot y_2=x_1\cdot (x_2*_2y_1)\cdot y_2=(x_1\cdot x_2)*_2(y_1\cdot y_2).
	\end{align*}
	By bilinearity, $*_1=*_2$ on $A^2$.
\end{proof}

In view of \cref{*_1=*_2-on-generators}, it suffices to know how an $\id$-matching structure on $UT_n(K)$ acts on the generators $e_{i,i+1}$, $1\le i<n$.

	\begin{lem}\label{e_ii+1*e_jj+1-possible-cases}
		Let $*$ be a bilinear product on $UT_n(K)$ satisfying \cref{matching-ass-identity} with $\sg=\id$. Given arbitrary $1\le i<n$ and $1\le j<n$, we have
		\begin{enumerate}
			\item\label{e_i_i+1*e_j_j+1-1<i<j+1<n} if $1<i<j+1<n$, then $e_{i,i+1}*e_{j,j+1}\in\spn_K\{e_{1,j+1},e_{1n},e_{in},e_{i,j+1}\}$;
			\item\label{e_i_i+1*e_j_j+1-1=i<j+1<n} if $1=i<j+1<n$, then $e_{i,i+1}*e_{j,j+1}\in\spn_K\{e_{1,j+1},e_{1n}\}$;
			\item\label{e_i_i+1*e_j_j+1-1<i<j+1=n} if $1<i<j+1=n$, then $e_{i,i+1}*e_{j,j+1}\in\spn_K\{e_{1n},e_{in}\}$;
			\item\label{e_i_i+1*e_j_j+1-1=i<j+1=n} if $1=i<j+1=n$, then $e_{i,i+1}*e_{j,j+1}\in\spn_K\{e_{1n}\}$;
			\item if $i\ge j+1$, then $e_{i,i+1}*e_{j,j+1}\in\spn_K\{e_{1,j+1},e_{1n},e_{in}\}$.
		\end{enumerate}
	\end{lem}
\begin{proof}
	Let $1<k<l$, $k\ne i$. Then
	\begin{align*}
		(e_{i,i+1}*e_{j,j+1})(k,l)=(e_{1k}\cdot(e_{i,i+1}*e_{j,j+1}))(1,l)=((e_{1k}\cdot e_{i,i+1})*e_{j,j+1})(1,l)=0.
	\end{align*}
	Similarly, for $k<l<n$, $l\ne j+1$, one has
	\begin{align*}
		(e_{i,i+1}*e_{j,j+1})(k,l)=((e_{i,i+1}*e_{j,j+1})\cdot e_{ln})(k,n)=(e_{i,i+1}*(e_{j,j+1}\cdot e_{ln}))(k,n)=0.
	\end{align*}
	Thus,
	\begin{align*}
		e_{i,i+1}*e_{j,j+1}\in
		\begin{cases}
			\spn_K\{e_{1,j+1},e_{1n},e_{in},e_{i,j+1}\}, & i<j+1,\\
			\spn_K\{e_{1,j+1},e_{1n},e_{in}\}, & i\ge j+1.
		\end{cases}
	\end{align*} 
Observe that in the case $i<j+1$ some of the elements $e_{1,j+1}$, $e_{1n}$, $e_{in}$, $e_{i,j+1}$ may coincide. More precisely, we have $4$ different subcases \cref{e_i_i+1*e_j_j+1-1<i<j+1<n,e_i_i+1*e_j_j+1-1=i<j+1<n,e_i_i+1*e_j_j+1-1<i<j+1=n,e_i_i+1*e_j_j+1-1=i<j+1=n} listed above. 
\end{proof}

We are going to show that the product $e_{i,i+1}*e_{j,j+1}$ can indeed be an arbitrary element of the corresponding span from \cref{e_ii+1*e_jj+1-possible-cases}. To this end, we need one more general lemma. For a non-unital associative algebra $A$ denote by $A^1$ its unitization. As a $K$-space, $A^1=A\oplus K$ and $(a,k)(b,l)=(ab+kb+la,kl)$ for all $(a,k),(b,l)\in A^1$. Then $A$ is an ideal in $A^1$. For any $a\in A$, the $K$-space $A^1a$ (resp. $aA^1$) is the left (resp. right) ideal of $A$ generated by $a$.

\begin{lem}\label{id-matching-structure-determined-by-a*b=c}
	Let $A$ be an associative algebra and $a,b,c\in A$ satisfying the following properties:
	\begin{enumerate}
		\item\label{xa=ya=>xc=yc-and-bx=by=>cx=cy} for all $x,y\in A^1$: $xa=ya\impl xc=yc$ and $bx=by\impl cx=cy$;
		\item\label{A=A^1a-oplus-I=bA^1-oplus-J} there is a left ideal $I$ in $A$ such that $A=A^1a\oplus I$ (as left $A$-modules) and there is a right ideal $J$ in $A$ such that $A=bA^1\oplus J$ (as right $A$-modules). 
	\end{enumerate}
Then, given $x,y\in A$, write 
\begin{align}\label{x=pa+u-and-y=bq+v}
	x=pa+u\text{ and }y=bq+v
\end{align}
with $p,q\in A^1$ and $u\in I$, $v\in J$. The formula
\begin{align}\label{(xa+u)*(by+v)-is-xcy}
	x*y=pcq
\end{align}
defines a bilinear product $*$ on $A$. Moreover it satisfies \cref{matching-ass-identity} with $\sg=\id$.
\end{lem}
\begin{proof}
	First of all, $*$ is well-defined. Indeed, by \cref{A=A^1a-oplus-I=bA^1-oplus-J} any element of $A$ can be written as $pa+u$ for some $p\in A^1$, $u\in I$ and as $bq+v$ for some $q\in A^1$, $v\in J$. Moreover, if $pa+u=p'a+u'$ with $p,p'\in A^1$ and $u,u'\in I$, then $pa=p'a$, whence $pc=p'c$ by \cref{xa=ya=>xc=yc-and-bx=by=>cx=cy}. Similarly, if $bq+v=bq'+v'$ with $q,q'\in A^1$ and $v,v'\in J$, then $bq=bq'$, whence $cq=cq'$ by \cref{xa=ya=>xc=yc-and-bx=by=>cx=cy}. Hence, $pcq=p'cq=p'cq'$, so the right-hand side of \cref{(xa+u)*(by+v)-is-xcy} depends only on $x$ and $y$, but not on the choice of $p$ and $q$ in \cref{x=pa+u-and-y=bq+v}. 
	
	It is clear that $*$ is bilinear. Furthermore, for any $x,y\in A$ as in \cref{x=pa+u-and-y=bq+v} and $z\in A$ we have $zx=(zp)a+zu$, where $zu\in I$, so $(zx)*y=(zp)cq=z(pcq)=z(x*y)$. Similarly, $yz=b(qz)+vz$, where $vz\in J$, so $x*(yz)=pc(qz)=(pcq)z=(x*y)z$.
\end{proof}

Observe that $UT_n(K)^1$ is isomorphic to the subalgebra $UT_n(K)+K\dl$ of $T_n(K)$, where $\dl$ denotes the identity element of $T_n(K)$. 
\begin{lem}\label{UT_n(K)^1e_ii+1-and-e_ii+1UT_n(K)^1}
	Let $1\le i<n$. Then
	\begin{enumerate}
		\item\label{UT_n(K)=UT_n(K)^1e_ii+1-oplus-UT_n(K)(dl-e_i+1)} $UT_n(K)=UT_n(K)^1\cdot e_{i,i+1}\oplus UT_n(K)\cdot (\dl-e_{i+1})$;
		\item\label{UT_n(K)=e_ii+1UT_n(K)^1-oplus-(dl-e_i)UT_n(K)} $UT_n(K)=e_{i,i+1}\cdot UT_n(K)^1\oplus (\dl-e_i)\cdot UT_n(K)$.
	\end{enumerate}
\end{lem}
\begin{proof}
It is clear that 
\begin{align*}
	UT_n(K)^1\cdot e_{i,i+1}&=\spn_K\{e_{i,i+1}\}+\spn_K\{e_{k,i+1}\mid k<i\}
	=\spn_K\{e_{k,i+1}\mid k<i+1\}=UT_n(K)\cdot e_{i+1}.
\end{align*}
Hence, 
\begin{align*}
	UT_n(K)=UT_n(K)\cdot\dl=UT_n(K)^1\cdot e_{i,i+1}+UT_n(K)\cdot(\dl-e_{i+1}).
\end{align*}
Since $e_{i+1}$ and $\dl-e_{i+1}$ are orthogonal idempotents of $T_n(K)$, the sum is direct, so we have \cref{UT_n(K)=UT_n(K)^1e_ii+1-oplus-UT_n(K)(dl-e_i+1)}. The proof of \cref{UT_n(K)=e_ii+1UT_n(K)^1-oplus-(dl-e_i)UT_n(K)} is similar.
\end{proof}

\begin{lem}\label{fe_ii+1=ge_ii+1-and-e_jj+1f=e_jj+1g}
	Let $1\le i<n$ and $1<j\le n$. Then
	\begin{enumerate}
		\item\label{e_inUT_n(K)=UT_n(K)e_1j=0} $e_{in}\cdot UT_n(K)=UT_n(K)\cdot e_{1j}=\{0\}$;
		\item\label{fe_ii+1=ge_ii+1=>fe_ik=ge_ik-and-e_jj+1f=e_jj+1g=>e_lj+1f=e_lj+1g} given $f,g\in UT_n(K)$, we have 
		\begin{align*}
			f\cdot e_{i,i+1}=g\cdot e_{i,i+1}&\impl f\cdot e_{ik}=g\cdot e_{ik}\text{ for all }k>i,\\
			e_{j,j+1}\cdot f=e_{j,j+1}\cdot g&\impl e_{l,j+1}\cdot f=e_{l,j+1}\cdot g\text{ for all }l<j+1.
		\end{align*}
	\end{enumerate}
\end{lem}
\begin{proof}
	Item \cref{e_inUT_n(K)=UT_n(K)e_1j=0} is obvious and \cref{fe_ii+1=ge_ii+1=>fe_ik=ge_ik-and-e_jj+1f=e_jj+1g=>e_lj+1f=e_lj+1g} follows from the fact that $e_{ik}=e_{i,i+1}\cdot e_{i+1,k}$ for all $k>i+1$ and $e_{l,j+1}=e_{lj}\cdot e_{j,j+1}$ for all $l<j$.
\end{proof}

\begin{lem}\label{structures-M^id1_ij}
	The following bilinear products $*$ on $UT_n(K)$ satisfy \cref{matching-ass-identity} with $\sg=\id$:\footnote{Here and below we usually omit the zero products in the multiplication table.}
	\begin{enumerate}
		\item\label{e_ii+1*e_jl=e_1l} $\M^{\id,1}_{i,j} (1<i\le n, 1\le j<n):$ $e_{i-1,i}*e_{jl}=e_{1l}$ for all $j<l\le n$;
		\item\label{e_ki*e_jj+1=e_kn} $\M^{\id,2}_{i,j} (1<i\le n, 1\le j<n, (i,j)\ne (2,n-1)):$ $e_{ki}*e_{j,j+1}=e_{kn}$ for all $1\le k<i$;
		\item\label{e_ki*e_jl=e_kl} $\M^{\id,3}_{i,j} (2<i\le j+1<n):$ $e_{ki}*e_{jl}=e_{kl}$ for all $1\le k<i$ and $j<l\le n$;
		\item\label{e_ki*e_jl=e_1n} $\M^{\id,4}_{i,j} (2<i\le n, 1\le j<n-1):$ $e_{i-1,i}*e_{j,j+1}=e_{1n}$.
	\end{enumerate} 
	Moreover, \cref{e_ii+1*e_jl=e_1l,e_ki*e_jj+1=e_kn,e_ki*e_jl=e_kl,e_ki*e_jl=e_1n} are linearly independent.
\end{lem}
\begin{proof}
	 All the products \cref{e_ii+1*e_jl=e_1l,e_ki*e_jj+1=e_kn,e_ki*e_jl=e_kl,e_ki*e_jl=e_1n} are determined by a fixed triple $(a,b,c)$ of elements of $UT_n(K)$ as in \cref{id-matching-structure-determined-by-a*b=c} with $A=UT_n(K)$, whose conditions are satisfied in view of \cref{UT_n(K)^1e_ii+1-and-e_ii+1UT_n(K)^1,fe_ii+1=ge_ii+1-and-e_jj+1f=e_jj+1g}. The elements $a=e_{i-1,i}$ and $b=e_{j,j+1}$ are common for all the products, while $c=e_{1,j+1}$ in \cref{e_ii+1*e_jl=e_1l}, $c=e_{i-1,n}$ in \cref{e_ki*e_jj+1=e_kn}, $c=e_{i-1,j+1}$ in \cref{e_ki*e_jl=e_kl} and $c=e_{1n}$ in \cref{e_ki*e_jl=e_1n}. The ideals $I$ and $J$ are $UT_n(K)\cdot (\dl-e_i)$ and $(\dl-e_j)\cdot UT_n(K)$, respectively.
	 
	 Indeed, for \cref{e_ii+1*e_jl=e_1l} we have $e_{i-1,i}=\dl\cdot e_{i-1,i}+0$ and $e_{jl}=e_{j,j+1}\cdot q+0$, where $q$ is either $\dl$ (if $l=j+1$) or $e_{j+1,l}$ (if $l>j+1$). Then $e_{i-1,i}*e_{jl}=\dl\cdot e_{1,j+1}\cdot q=e_{1l}$ by \cref{(xa+u)*(by+v)-is-xcy}. Now, let $x<y$ and $z<w$. If $y\ne i$, then $e_{xy}\in I$, so $e_{xy}*e_{zw}=0$ by \cref{(xa+u)*(by+v)-is-xcy}. Similarly, if $z\ne j$, then $e_{zw}\in J$, so $e_{xy}*e_{zw}=0$ by \cref{(xa+u)*(by+v)-is-xcy}. Finally, if $y=i$, $z=j$ and $x<i-1$, we have $e_{xy}=e_{x,i-1}\cdot e_{i-1,i}+0$ and $e_{zw}=e_{j,j+1}\cdot q+0$, where $q$ is either $\dl$ (if $w=j+1$) or $e_{j+1,w}$ (if $w>j+1$). In any case, $e_{xy}*e_{zw}=e_{x,i-1}\cdot e_{1,j+1}\cdot q=0$ by \cref{(xa+u)*(by+v)-is-xcy} because $e_{x,i-1}\cdot e_{1,j+1}=0$.
	 
	 Furthermore, for \cref{e_ki*e_jl=e_kl} we have $e_{ki}=p\cdot e_{i-1,i}+0$ and $e_{jl}=e_{j,j+1}\cdot q+0$, where $p$ is either $\dl$ (if $k=i-1$) or $e_{k,i-1}$ (if $k<i-1$) and $q$ is either $\dl$ (if $l=j+1$) or $e_{j+1,l}$ (if $l>j+1$). Then $e_{ki}*e_{jl}=p\cdot e_{i-1,j+1}\cdot q=e_{kl}$ by \cref{(xa+u)*(by+v)-is-xcy}. Now, let $x<y$ and $z<w$. If $y\ne i$, then $e_{xy}\in I$, so $e_{xy}*e_{zw}=0$ by \cref{(xa+u)*(by+v)-is-xcy}. Similarly, if $z\ne j$, then $e_{zw}\in J$, so $e_{xy}*e_{zw}=0$ by \cref{(xa+u)*(by+v)-is-xcy}. 
	 
	 The proof for \cref{e_ki*e_jj+1=e_kn} is similar\footnote{Alternatively, one can observe that \cref{e_ki*e_jj+1=e_kn} is obtained from \cref{e_ii+1*e_jl=e_1l} using the antiautomorphism $\phi(e_{kl})=e_{n-l+1,n-k+1}$ of $UT_n(K)$ as in \cref{antiauto=>structure-of-the-same-type}.} to that for \cref{e_ii+1*e_jl=e_1l}, and the proof for \cref{e_ki*e_jl=e_1n} is obvious because $e_{1n}\in\Ann(UT_n(K),\cdot)$.
	 
	 We impose the restriction $(i,j)\ne (2,n-1)$ in \cref{e_ki*e_jj+1=e_kn}, because $\M^{\id,1}_{2,n-1}=\M^{\id,2}_{2,n-1}$. Similarly, we exclude $i=2$ and $j=n-1$ from \cref{e_ki*e_jl=e_kl,e_ki*e_jl=e_1n}, because $\M^{\id,3}_{2,j}=\M^{\id,1}_{2,j}$, $\M^{\id,3}_{i,n-1}=\M^{\id,2}_{i,n-1}$, $\M^{\id,4}_{2,j}=\M^{\id,2}_{2,j}$ and $\M^{\id,4}_{i,n-1}=\M^{\id,1}_{i,n-1}$. With these restrictions on indices, the products \cref{e_ii+1*e_jl=e_1l,e_ki*e_jj+1=e_kn,e_ki*e_jl=e_kl,e_ki*e_jl=e_1n} are linearly independent.
\end{proof}

\begin{rem}
	There are $(n-1)^2$ products of the form \cref{e_ii+1*e_jl=e_1l}, $(n-1)^2-1$ products of the form \cref{e_ki*e_jj+1=e_kn}, $\frac 12(n-3)(n-2)$ products of the form \cref{e_ki*e_jl=e_kl} and $(n-2)^2$ products of the form \cref{e_ki*e_jl=e_1n}. In total, there are $\frac 72(n-2)(n-1)+1$ products of the form \cref{e_ii+1*e_jl=e_1l,e_ki*e_jj+1=e_kn,e_ki*e_jl=e_kl,e_ki*e_jl=e_1n}.
\end{rem}

\begin{rem}\label{M^id4_ij-totally-comp}
	The structure $\M^{\id,4}_{i,j}$ is a particular case of \cref{annihilator-valued-structure}, because $e_{i-1,i}$ and $e_{j,j+1}$ belong to a complement of $UT_n(K)^2$ and $e_{1n}\in \Ann(UT_n(K),\cdot)$. In particular, $\M^{\id,4}_{i,j}$ is totally compatible with $\cdot$.
\end{rem}

\begin{thrm}\label{id-matching-UT_n-full-descr}
	Let $n>2$. Then the $\id$-matching structures on $UT_n(K)$ are exactly the associative products on $UT_n(K)$ that are linear combinations of the products \cref{e_ii+1*e_jl=e_1l,e_ki*e_jj+1=e_kn,e_ki*e_jl=e_kl,e_ki*e_jl=e_1n} of \cref{structures-M^id1_ij}.
\end{thrm}
\begin{proof}
	Any associative linear combination of the products \cref{e_ii+1*e_jl=e_1l,e_ki*e_jj+1=e_kn,e_ki*e_jl=e_kl,e_ki*e_jl=e_1n} is an $\id$-matching structure on $UT_n(K)$ by \cref{structures-M^id1_ij}. Conversely, let $*$ be an $\id$-matching structure on $UT_n(K)$. It follows from \cref{e_ii+1*e_jj+1-possible-cases} that $*$ is a linear combination of \cref{e_ii+1*e_jl=e_1l,e_ki*e_jj+1=e_kn,e_ki*e_jl=e_kl,e_ki*e_jl=e_1n}, when restricted to the pairs $(e_{i,i+1},e_{j,j+1})$, $1\le i<n$, $1\le j<n$. Since $e_{i,i+1}$, $1\le i<n$, generate $UT_n(K)$ as a $K$-algebra, then $*$ coincides with a linear combination of \cref{e_ii+1*e_jl=e_1l,e_ki*e_jj+1=e_kn,e_ki*e_jl=e_kl,e_ki*e_jl=e_1n} on $UT_n(K)^2$ by \cref{*_1=*_2-on-generators}.
\end{proof}

\subsection{The complete classification of $\id$-matching structures on $UT_3(K)$}\label{sec-id-match-n=3}

We are going to apply \cref{id-matching-UT_n-full-descr} to classify the $\id$-matching structures on $UT_3(K)$ up to an isomorphism. The following corollary is just a specification of \cref{id-matching-UT_n-full-descr} to the case $n=3$. Observe that there are no products of the form \cref{e_ki*e_jl=e_kl} on $UT_3(K)$.

\begin{cor}\label{structures-M^id1_ij-n=3}
	The $\id$-matching structures on $UT_3(K)$ are exactly the associative products on $UT_3(K)$ that are linear combinations of the products:
	\begin{multicols}{2}
	\begin{enumerate}
		\item\label{M^id1_21-n=3} $\M^{\id,1}_{2,1}:$ $e_{12}*e_{12}=e_{12}$, $e_{12}*e_{13}=e_{13}$;
		\item\label{M^id1_22-n=3} $\M^{\id,1}_{2,2}:$ $e_{12}*e_{23}=e_{13}$;
		\item\label{M^id1_31-n=3} $\M^{\id,1}_{3,1}:$ $e_{23}*e_{12}=e_{12}$, $e_{23}*e_{13}=e_{13}$;
		\item\label{M^id1_32-n=3} $\M^{\id,1}_{3,2}:$ $e_{23}*e_{23}=e_{13}$;
		\item\label{M^id2_21-n=3} $\M^{\id,2}_{2,1}:$ $e_{12}*e_{12}=e_{13}$;
		\item\label{M^id2_31-n=3} $\M^{\id,2}_{3,1}:$ $e_{13}*e_{12}=e_{13}$, $e_{23}*e_{12}=e_{23}$;
		\item\label{M^id2_32-n=3} $\M^{\id,2}_{3,2}:$ $e_{13}*e_{23}=e_{13}$, $e_{23}*e_{23}=e_{23}$;
		%\item\label{e_ki*e_jl=e_kl} $\M^{\id,3}_{i,j}$ ($2<i\le j+1<3$): $e_{ki}*e_{jl}=e_{kl}$ for all $j<l\le n$;
		\item\label{M^id4_31-n=3} $\M^{\id,4}_{3,1}:$ $e_{23}*e_{12}=e_{13}$,
	\end{enumerate}
	\end{multicols}
i.e. the associative products of the form
\begin{align}\label{id-matching-UT_3(K)-general-form}
	\begin{array}{rclrcl}
		e_{12}*e_{12}&=&\lb_1 e_{12}+\lb_5 e_{13},& e_{13}*e_{23}&=&\lb_7 e_{13},\\
		e_{12}*e_{13}&=&\lb_1 e_{13},& e_{23}*e_{12}&=&\lb_3 e_{12}+\lb_6 e_{23}+\lb_8 e_{13},\\
		e_{12}*e_{23}&=&\lb_2 e_{13},& e_{23}*e_{13}&=&\lb_3 e_{13},\\
		e_{13}*e_{12}&=&\lb_6 e_{13},& e_{23}*e_{23}&=&\lb_4 e_{13}+\lb_7 e_{23}
	\end{array}
\end{align}
for some $\{\lb_i\}_{i=1}^8\sst K$.
\end{cor}

Denote by $[a,b,c]=(a*b)*c-a*(b*c)$ the \textit{associator} of $a,b,c\in UT_3(K)$ under $*$. The proofs of the following two lemmas are straightforward, so they will be omitted.
\begin{lem}\label{associators-UT_3-id-match}
	We have the following non-trivial associators of the basis elements $\{e_{12},e_{23},e_{13}\}$ of $UT_3(K)$ under \cref{id-matching-UT_3(K)-general-form}:
	\begin{align}
		[e_{12}, e_{12}, e_{12}]&= \lb_5(\lb_6-\lb_1)e_{13},\label{[e_12_e_12_e_12]}\\
		[e_{12}, e_{12}, e_{23}]&= \lb_5\lb_7e_{13},\label{[e_12_e_12_e_23]}\\
		[e_{12}, e_{23}, e_{12}]&= -\lb_1\lb_3e_{12} - (\lb_1\lb_8 + \lb_5\lb_3)e_{13},\label{[e_12_e_23_e_12]}\\
		[e_{12}, e_{23}, e_{23}]&= -\lb_1\lb_4e_{13},\label{[e_12_e_23_e_23]}\\
		[e_{12}, e_{23}, e_{13}]&= -\lb_1\lb_3e_{13},\notag\\
		[e_{23}, e_{12}, e_{12}]&= \lb_6\lb_3e_{12} + \lb_6(\lb_6-\lb_1)e_{23} + \lb_8(2\lb_6-\lb_1)e_{13},\label{[e_23_e_12_e_12]}\\
		[e_{23}, e_{12}, e_{23}]&= \lb_7\lb_6e_{23} + (\lb_7\lb_8 + \lb_6\lb_4)e_{13},\label{[e_23_e_12_e_23]}\\
		[e_{23}, e_{12}, e_{13}]&=-[e_{13}, e_{23}, e_{12}]= \lb_6\lb_3e_{13},\notag\\
		[e_{23}, e_{23}, e_{12}]&= \lb_3(\lb_7 - \lb_3)e_{12} - \lb_6\lb_3e_{23} + \lb_8(\lb_7 - 2\lb_3)e_{13},\label{[e_23_e_23_e_12]}\\
		[e_{23}, e_{23}, e_{23}]&= \lb_4(\lb_7 - \lb_3)e_{13},\label{[e_23_e_23_e_23]}\\
		[e_{23}, e_{23}, e_{13}]&= \lb_3(\lb_7 - \lb_3)e_{13},\notag\\
		[e_{13}, e_{12}, e_{12}]&= \lb_6(\lb_6-\lb_1)e_{13},\notag\\
		[e_{13}, e_{12}, e_{23}]&= \lb_7\lb_6e_{13}.\notag
	\end{align}
\end{lem}

\begin{lem}\label{id-matching-UT_3(K)-new-base}
	Let $\phi\in\Aut(UT_3(K))$ as in \cref{aut-UT_3-on-e_12_e_23_e_13} and $*$ be given by \cref{id-matching-UT_3(K)-general-form}. Then the corresponding product $\star$ on $UT_3(K)$ defined by \cref{a-star-b=phi(phi-inv(a)*phi-inv(b))} is as follows:
	\begin{align*}
			\begin{array}{rclrcl}
				e_{12} \star e_{12} &=& a_{11}^{-1}\lb_1e_{12} + a_{11}^{-2}(a_{11}a_{22}\lb_5 - a_{31}\lb_6)e_{13}, & e_{23} \star e_{23} &=& a_{22}^{-1}\lb_7e_{23} + a_{22}^{-2}(a_{11}a_{22}\lb_4 - a_{32}\lb_3)e_{13},\\
				e_{12} \star e_{23} &=& a_{11}^{-1} a_{22}^{-1} (a_{11}a_{22}\lb_2 - a_{32}\lb_1 - a_{31}\lb_7)e_{13}, & e_{23} \star e_{13} &=& a_{22}^{-1}\lb_3e_{13},\\
				e_{12} \star e_{13} &=& a_{11}^{-1}\lb_1e_{13}, & e_{13} \star e_{12} &=& a_{11}^{-1}\lb_6e_{13},\\
				e_{23} \star e_{12} &=& a_{22}^{-1}\lb_3e_{12} + a_{11}^{-1}\lb_6e_{23} + \lb_8e_{13}, & e_{13} \star e_{23} &=& a_{22}^{-1}\lb_7e_{13}.\\
			\end{array}
	\end{align*} 
\end{lem}

\begin{prop}\label{id-matching-UT_3-full-descr}
	Any $\id$-matching structure on $UT_3(K)$ is isomorphic to exactly one of the following structures:
	\begin{enumerate}
		\item\label{star-for-lb_3-ne-0} $e_{13}\star e_{23} = e_{13},\ e_{23}\star e_{12} = e_{12},\ e_{23}\star e_{13} = e_{13},\ e_{23}\star e_{23} = e_{23}$;
		\item\label{star-for-lb_3=0-lb_7-ne-0-lb_1-ne-0} $e_{12}\star e_{12} = e_{12},\ e_{12}\star e_{13} = e_{13},\ e_{13}\star e_{23} = e_{13},\ e_{23}\star e_{23} = e_{23}$;
		\item\label{star-for-lb_3=0-lb_7-ne-0-lb_1-is-0} $e_{13}\star e_{23} = e_{13},\ e_{23}\star e_{23} = e_{23}$;
		\item\label{star-for-lb_3=lb_7=0-lb_4-ne-0} $e_{12}\star e_{12} = \af e_{13},\ e_{12}\star e_{23} = \bt e_{13},\ e_{23}\star e_{12} = \gm e_{13},\ e_{23}\star e_{23} = e_{13}$;
		\item\label{star-for-lb_3=lb_7=lb_4=0-lb_6-ne-0} $e_{12}\star e_{12} = e_{12},\ e_{12}\star e_{13} = e_{13},\ e_{23}\star e_{12} = e_{23},\ e_{13}\star e_{12} = e_{13}$;
		\item\label{star-for-lb_3=lb_7=lb_4=0-lb_6=0-lb_1-ne-0} $e_{12}\star e_{12} = e_{12},\ e_{12}\star e_{13} = e_{13}$;
		\item\label{star-for-lb_3=lb_7=lb_4=0-lb_6=lb_1=0-lb_5-ne-0} $e_{12}\star e_{12} = e_{13},\ e_{12} \star e_{23} = \af e_{13},\ e_{23} \star e_{12} = \bt e_{13}$;
		\item\label{star-for-lb_3=lb_7=lb_4=0-lb_6=lb_1=0-lb_5=0} $e_{12} \star e_{23} = \af e_{13},\ e_{23} \star e_{12} = \bt e_{13}$.
	\end{enumerate}
\end{prop}
\begin{proof}
	Let $*$ be an $\id$-matching structure on $UT_3(K)$. By \cref{structures-M^id1_ij-n=3} the product $*$ is of the form \cref{id-matching-UT_3(K)-general-form}. 
	
	\textit{Case 1.} $\lb_3\ne 0$. It follows from \cref{[e_12_e_23_e_12]} that $\lb_1=\lb_5=0$. Furthermore, by \cref{[e_23_e_23_e_12]} we have $\lb_7=\lb_3$ and $\lb_6=\lb_8=0$. Let $\phi\in\Aut(UT_3(K))$ given by \cref{aut-UT_3-on-e_12_e_23_e_13}. Then by \cref{id-matching-UT_3(K)-new-base} the automorphism $\phi$ defines the following isomorphic structure:
	\begin{align*}
		\begin{array}{rclrcl}
			e_{12} \star e_{23} &=& a_{11}^{-1} a_{22}^{-1} (a_{11}a_{22}\lb_2 - a_{31}\lb_3)e_{13}, & e_{23} \star e_{13} &=& a_{22}^{-1}\lb_3e_{13},\\
			e_{13} \star e_{23} &=& a_{22}^{-1}\lb_3e_{13}, & e_{23} \star e_{23} &=& a_{22}^{-1}\lb_3e_{23} + a_{22}^{-2}(a_{11}a_{22}\lb_4 - a_{32}\lb_3)e_{13}.\\
			e_{23} \star e_{12} &=& a_{22}^{-1}\lb_3e_{12},\\
		\end{array}
	\end{align*}
Choosing $a_{11}=1$, $a_{22}=\lb_3$, $a_{31}=\lb_2$ and $a_{32}=\lb_4$, we get \cref{star-for-lb_3-ne-0}.

\textit{Case 2.} $\lb_3=0$ and $\lb_7\ne 0$. Then \cref{[e_12_e_12_e_23]} implies $\lb_5=0$.
Furthermore, \cref{[e_23_e_12_e_23]} gives $\lb_6=\lb_8=0$. Finally, \cref{[e_23_e_23_e_23]} yields $\lb_4=0$. Let $\phi\in\Aut(UT_3(K))$ given by \cref{aut-UT_3-on-e_12_e_23_e_13}. Then by \cref{id-matching-UT_3(K)-new-base} the automorphism $\phi$ defines the following isomorphic structure
\begin{align*}
	\begin{array}{rclrcl}
		e_{12} \star e_{12} &=& a_{11}^{-1}\lb_1e_{12},& e_{13} \star e_{23} &=& a_{22}^{-1}\lb_7e_{13},\\
		e_{12} \star e_{13} &=& a_{11}^{-1}\lb_1e_{13}, & e_{23} \star e_{23} &=& a_{22}^{-1}\lb_7e_{23}.\\
		e_{12} \star e_{23} &=& a_{11}^{-1} a_{22}^{-1} (a_{11}a_{22}\lb_2 - a_{32}\lb_1 - a_{31}\lb_7)e_{13},
	\end{array}
\end{align*}
If $\lb_1\ne 0$, then choosing $a_{11}=\lb_1$, $a_{22}=\lb_7$, $a_{31} = \lb_1\lb_2$ and $a_{32}=0$, we get the structure \cref{star-for-lb_3=0-lb_7-ne-0-lb_1-ne-0}. Otherwise, choosing $a_{11}=1$, $a_{22}=\lb_7$, $a_{31}=\lb_2$  and $a_{32}=0$, we get \cref{star-for-lb_3=0-lb_7-ne-0-lb_1-is-0}.

\textit{Case 3.} $\lb_3=\lb_7=0$ and $\lb_4\ne 0$. Then \cref{[e_12_e_23_e_23]} implies $\lb_1=0$ and \cref{[e_23_e_12_e_23]} gives $\lb_6=0$. Let $\phi\in\Aut(UT_3(K))$ given by \cref{aut-UT_3-on-e_12_e_23_e_13}. Then by \cref{id-matching-UT_3(K)-new-base} the automorphism $\phi$ defines the following isomorphic structure
\begin{align*}
		e_{12} \star e_{12} = a_{11}^{-1}a_{22}\lb_5e_{13},\  e_{23} \star e_{12} = \lb_8e_{13},\
		e_{12} \star e_{23} = \lb_2e_{13},\  e_{23} \star e_{23} = a_{11}a_{22}^{-1}\lb_4e_{13}.
\end{align*}
Choosing $a_{11}=1$ and $a_{22}=\lb_4$, we get the family \cref{star-for-lb_3=lb_7=0-lb_4-ne-0} of pairwise non-isomorphic structures.

\textit{Case 4.} $\lb_3=\lb_7=\lb_4=0$ and $\lb_6\ne 0$. Then \cref{[e_23_e_12_e_12]} implies $\lb_1=\lb_6$ and $\lb_8=0$. Let $\phi\in\Aut(UT_3(K))$ given by \cref{aut-UT_3-on-e_12_e_23_e_13}. Then by \cref{id-matching-UT_3(K)-new-base} the automorphism $\phi$ defines the following isomorphic structure
\begin{align*}
	\begin{array}{rclrcl}
		e_{12} \star e_{12} &=& a_{11}^{-1}\lb_6e_{12} + a_{11}^{-2}(a_{11}a_{22}\lb_5 - a_{31}\lb_6)e_{13}, & e_{23} \star e_{12} &=& a_{11}^{-1}\lb_6e_{23},\\
		e_{12} \star e_{23} &=& (\lb_2 - a_{11}^{-1} a_{22}^{-1}a_{32}\lb_6)e_{13}, & e_{13} \star e_{12} &=& a_{11}^{-1}\lb_6e_{13}.\\
		e_{12} \star e_{13} &=& a_{11}^{-1}\lb_6e_{13},
	\end{array}
\end{align*}
Choosing $a_{11}=\lb_6$, $a_{22}=1$, $a_{31}=\lb_5$ and $a_{32}=\lb_2$, we get the structure \cref{star-for-lb_3=lb_7=lb_4=0-lb_6-ne-0}.

\textit{Case 5.} $\lb_3=\lb_7=\lb_4=\lb_6=0$ and $\lb_1\ne 0$. Then \cref{[e_12_e_12_e_12]} implies $\lb_5=0$ and \cref{[e_12_e_23_e_12]} implies $\lb_8=0$. Let $\phi\in\Aut(UT_3(K))$ given by \cref{aut-UT_3-on-e_12_e_23_e_13}. Then by \cref{id-matching-UT_3(K)-new-base} the automorphism $\phi$ defines the following isomorphic structure
\begin{align*}
	e_{12} \star e_{12} = a_{11}^{-1}\lb_1e_{12},\  e_{12} \star e_{13} = a_{11}^{-1}\lb_1e_{13},\ 
	e_{12} \star e_{23} = (\lb_2 - a_{11}^{-1} a_{22}^{-1}a_{32}\lb_1)e_{13}.
\end{align*}
Choosing $a_{11}=\lb_1$, $a_{22}=1$ and $a_{32}=\lb_2$, we get the structure \cref{star-for-lb_3=lb_7=lb_4=0-lb_6=0-lb_1-ne-0}.

\textit{Case 6.} $\lb_3=\lb_7=\lb_4=\lb_6=\lb_1=0$. Then $*$ is associative. Let $\phi\in\Aut(UT_3(K))$ given by \cref{aut-UT_3-on-e_12_e_23_e_13}. Then by \cref{id-matching-UT_3(K)-new-base} the automorphism $\phi$ defines the following isomorphic structure
\begin{align*}
	e_{12} \star e_{12} = a_{11}^{-1}a_{22}\lb_5e_{13},\ e_{12} \star e_{23} = \lb_2e_{13},\ e_{23} \star e_{12} = \lb_8e_{13}.
\end{align*}
If $\lb_5\ne 0$, then choosing $a_{11}=\lb_5$ and $a_{22}=1$, we get the family \cref{star-for-lb_3=lb_7=lb_4=0-lb_6=lb_1=0-lb_5-ne-0} of pairwise non-isomorphic structures. Otherwise, we get the family \cref{star-for-lb_3=lb_7=lb_4=0-lb_6=lb_1=0-lb_5=0} of pairwise non-isomorphic structures.

Conversely, each product of the list \cref{star-for-lb_3-ne-0,star-for-lb_3=0-lb_7-ne-0-lb_1-ne-0,star-for-lb_3=0-lb_7-ne-0-lb_1-is-0,star-for-lb_3=lb_7=0-lb_4-ne-0,star-for-lb_3=lb_7=lb_4=0-lb_6-ne-0,star-for-lb_3=lb_7=lb_4=0-lb_6=0-lb_1-ne-0,star-for-lb_3=lb_7=lb_4=0-lb_6=lb_1=0-lb_5-ne-0,star-for-lb_3=lb_7=lb_4=0-lb_6=lb_1=0-lb_5=0} is a particular case of \cref{id-matching-UT_3(K)-general-form}, and it is associative by \cref{associators-UT_3-id-match}, so it is an $\id$-matching structure on $UT_3(K)$ by \cref{structures-M^id1_ij-n=3}. Moreover, it is directly seen by \cref{id-matching-UT_3(K)-new-base} that all the listed structures are pairwise non-isomorphic.
\end{proof}

	\section{$(12)$-matching structures on $UT_n(K)$}\label{sec-(12)-match}
	
	\subsection{Some classes of $(12)$-matching structures on $UT_n(K)$}\label{sec-(12)-matching-classes}
	We first point out the structures that will appear in the classification below. Clearly, $\M^{\id,4}_{i,j}$ from \cref{structures-M^id1_ij} is $(12)$-matching with $\cdot$ for all $1<i\le n$ and $1\le j<n$ by \cref{M^id4_ij-totally-comp}.
	%The following structure is of the form a*b=c, where a,b belong to a complement of A^2 and c belongs to Ann(A). For such * all the mixed monomials are zero.
%	\begin{rem}\label{ann-valued-structure-(12)}
%		For each fixed pair $1<i\le n$, $1\le j<n$ the bilinear product 
%		\begin{align}\label{e_i-1i*e_jj+1=e_1n}
%			e_{i-1,i}*e_{j,j+1}=e_{1n}
%		\end{align}
%		is a $(12)$-matching structure on $UT_n(K)$.
%		
%		Indeed, $*$ is a particular case of the structure \cref{annihilator-valued-structure}, because $e_{i-1,i}$ and $e_{j,j+1}$ belong to a complement of $UT_n(K)^2$ and $e_{1n}\in \Ann(UT_n(K),\cdot)$.
%	\end{rem}
	
	\begin{lem}
		For each fixed $1\le i<n$ the bilinear product
		\begin{align}\label{e_12*e_ii+1=e_1n-1-e_ii+1*e_n-1n=e_2n}
			e_{12}*e_{i,i+1}=e_{1,n-1},\ e_{i,i+1}*e_{n-1,n}=e_{2n}
		\end{align}
		is a $(12)$-matching structure on $UT_n(K)$.
	\end{lem}
	\begin{proof}
		Since $e_{12}, e_{i,i+1},e_{n-1,n}\in\spn_K\{e_{i,i+1}: 1\le i<n\}$, which is a complement of $UT_n(K)^2$, then $(a\cdot b)*c=a*(b\cdot c)=0$ for all $a,b,c\in UT_n(K)$.
			
		Now, observe that 
		\begin{align*}
			(e_{xy}*e_{zw})\cdot e_{uv}\ne 0\iff e_{xy}\cdot(e_{zw}*e_{uv})\ne 0\iff (e_{xy},e_{zw},e_{uv})=(e_{12},e_{i,i+1},e_{n-1,n}),
		\end{align*}
		in which case 
		\begin{align*}
			(e_{xy}*e_{zw})\cdot e_{uv}&=(e_{12}*e_{i,i+1})\cdot e_{n-1,n}=e_{1,n-1}\cdot e_{n-1,n}=e_{1n},\\
			e_{xy}\cdot(e_{zw}*e_{uv})&=e_{12}\cdot(e_{i,i+1}*e_{n-1,n})=e_{12}\cdot e_{2n}=e_{1n}.
		\end{align*}
		Thus, $(a*b)\cdot c=a\cdot(b*c)$ for all $a,b,c\in UT_n(K)$.
		
		For the associativity observe that $(a*b)*c=a*(b*c)=0$ for all $a,b,c\in UT_n(K)$, if $n>3$. Now, on $UT_3(K)$ we have $2$ structures of the form \cref{e_12*e_ii+1=e_1n-1-e_ii+1*e_n-1n=e_2n}, corresponding to $i=1$ and $i=2$. It is easily seen that they are isomorphic to the subalgebras of $(T_2(K),\cdot)$ generated by $\{e_1,e_{12}\}$ and $\{e_2,e_{12}\}$, respectively.
	\end{proof}
	
	\begin{lem}
		For each fixed $1\le i<j\le n$ with $j>i+1$ the bilinear product
		\begin{align}\label{e_ik*e_kj=e_ij-fixed-i-and-j}
			e_{ik}*e_{kj}=e_{1n},\ i<k<j,
		\end{align}
		is a $(12)$-matching structure on $UT_n(K)$.
	\end{lem}
	\begin{proof}
		Clearly, $(a*b)\cdot c=a\cdot(b*c)=0$ for all $a,b,c\in UT_n(K)$, because $e_{1n}\in \Ann(UT_n(K),\cdot)$. 
		
		Now, 
		\begin{align*}
			(e_{xy}\cdot e_{zw})*e_{uv}\ne 0&\iff e_{xy}\cdot e_{zw}=e_{ik}\text{ and }e_{uv}=e_{kj}\text{ for some }i<k<j\\
			&\iff(e_{xy},e_{zw},e_{uv})=(e_{iy},e_{yk},e_{kj})\text{ for some }i<k<j,
		\end{align*}
		in which case
		\begin{align*}
			(e_{xy}\cdot e_{zw})*e_{uv}=(e_{iy}\cdot e_{yk})*e_{kj}=e_{ik}*e_{kj}=e_{1n},\
			e_{xy}*(e_{zw}\cdot e_{uv})=e_{iy}*(e_{yk}\cdot e_{kj})=e_{iy}*e_{yj}=e_{1n}.
		\end{align*}
		
		Symmetrically, 
		\begin{align*}
			e_{xy}*(e_{zw}\cdot e_{uv})\ne 0&\iff e_{xy}=e_{ik}\text{ and }e_{zw}\cdot e_{uv}=e_{kj}\text{ for some }i<k<j\\
			&\iff(e_{xy},e_{zw},e_{uv})=(e_{ik},e_{kw},e_{wj})\text{ for some }i<k<j,
		\end{align*}
		in which case
		\begin{align*}
			e_{xy}*(e_{zw}\cdot e_{uv})=e_{ik}*(e_{kw}\cdot e_{wj})=e_{ik}*e_{kj}=e_{1n},\
			(e_{xy}\cdot e_{zw})*e_{uv}=(e_{ik}\cdot e_{kw})*e_{wj}=e_{iw}*e_{wj}=e_{1n}.
		\end{align*}
		
		For the associativity observe that $a*e_{1n}=e_{1n}*a=0$ for all $a\in UT_n(K)$, so $(a*b)*c=a*(b*c)=0$ for all $a,b,c\in UT_n(K)$.
	\end{proof}
	
	\begin{lem}
		The bilinear product
		\begin{align}\label{e_ik*e_kj=e_ij-(i_j)-ne-(1_n)}
			e_{ik}*e_{kj}=e_{ij}, 1\le i<k<j\le n,\ (i,j)\ne(1,n), 
		\end{align}
		is a $(12)$-matching structure on $UT_n(K)$.
	\end{lem}
	\begin{proof}
		The structure \cref{e_ik*e_kj=e_ij-(i_j)-ne-(1_n)} is the difference between two $(12)$-matching structures on $UT_n(K)$: the original product $\cdot$ on $UT_n(K)$ and the product \cref{e_ik*e_kj=e_ij-fixed-i-and-j} corresponding to $(i,j)=(1,n)$. Thus, \cref{e_ik*e_kj=e_ij-(i_j)-ne-(1_n)} satisfies \cref{matching-ass-identity} with $\sg=(12)$.
		
		For the associativity observe that 
		\begin{align*}
			(e_{xy}*e_{zw})*e_{uv}\ne 0\iff e_{xy}*(e_{zw}*e_{uv})\ne 0\iff y=z,\ w=u\text{ and }(x,v)\ne (1,n),
		\end{align*}
		in which case $(e_{xy}*e_{zw})*e_{uv}=e_{xy}*(e_{zw}*e_{uv})=e_{xv}$.
	\end{proof}

\subsection{The decomposition of a $(12)$-matching structure on $UT_n(K)$}
	
Let $*$ be a (not necessarily associative) bilinear product on $UT_n(K)$ satisfying \cref{matching-ass-identity} with $\sg=(12)$.

\subsubsection{The product $e_{ij}*e_{kl}$ with $j\ne k$}\label{sec-(12)-e_ij*e_kl-j-ne-k}

\begin{lem}\label{e_ij*e_kl=0-for-j-ne-k-and-j-i>1-or-l-k>1}
	 For all $1\le i<j\le n$ and $1\le k<l\le n$ if $j\ne k$ and either $j>i+1$ or $l>k+1$, then $e_{ij}*e_{kl}=0$.
\end{lem}
\begin{proof}
	Let $j\ne k$ and $j>i+1$. Choose $i<l<j$. Then $e_{ij}*e_{kl}=(e_{il}\cdot e_{lj})*e_{kl}=e_{il}*(e_{lj}\cdot e_{kl})=0$. The case $j\ne k$ and $l>k+1$ is analogous.
\end{proof}

\begin{lem}\label{e_ij*e_kl-in-span-e_1_n-1_e_1n}
	For all $1\le i\le n-1$ and $1\le j\le n-2$ if $i+1\ne j$, then $e_{i,i+1}*e_{j,j+1}\in\spn_K\{e_{1,n-1},e_{1n}\}$.
\end{lem}
\begin{proof}
	Since $1\le j\le n-2$, one has $e_{i,i+1}*e_{j,j+2}=0$ by \cref{e_ij*e_kl=0-for-j-ne-k-and-j-i>1-or-l-k>1}. Then $(e_{1l}*e_{i,i+1})\cdot e_{j,j+2}=e_{1l}\cdot (e_{i,i+1}*e_{j,j+2})=0$ for all $1<l\le n$. Hence, $(e_{1l}*e_{i,i+1})(s,j)=0$ for all $s<j$, so $e_{1l}\cdot (e_{i,i+1}*e_{j,j+1})=(e_{1l}*e_{i,i+1})\cdot e_{j,j+1}=0$. It follows that 
	\begin{align}\label{e_i_i+1*e_j_j+1-in-span-e_1k}
		e_{i,i+1}*e_{j,j+1}\in\spn_K\{e_{1k}: 1<k\le n\}.
	\end{align}
Now take $1<k<n-1$ with $k\ne j+1$. Then $e_{j,j+1}*e_{k,k+1}\in\spn_K\{e_{1t}: 1<t\le n\}$ by \cref{e_i_i+1*e_j_j+1-in-span-e_1k}, so $(e_{i,i+1}*e_{j,j+1})\cdot e_{k,k+1}=e_{i,i+1}\cdot(e_{j,j+1}*e_{k,k+1})=0$. This reduces $1<k\le n$ to $k\in\{n-1,n,j+1\}$ in \cref{e_i_i+1*e_j_j+1-in-span-e_1k}. If $j=n-2$, we are done. Otherwise, $j\le n-3$. Write $(e_{i,i+1}*e_{j,j+1})\cdot e_{j+1,j+2}=e_{i,i+1}\cdot(e_{j,j+1}*e_{j+1,j+2})$, whence
	\begin{align}\label{(e_i_i+1*e_j_j+1)(1_j+1)=(e_i_i+1-cdot-(e_j_j+1*e_j+1_j+2))(1_j+2)}
		(e_{i,i+1}*e_{j,j+1})(1,j+1)=(e_{i,i+1}\cdot(e_{j,j+1}*e_{j+1,j+2}))(1,j+2).
	\end{align}
	If $i>1$, then $(e_{i,i+1}*e_{j,j+1})(1,j+1)=0$ by \cref{(e_i_i+1*e_j_j+1)(1_j+1)=(e_i_i+1-cdot-(e_j_j+1*e_j+1_j+2))(1_j+2)}, and we are done. Otherwise, by \cref{(e_i_i+1*e_j_j+1)(1_j+1)=(e_i_i+1-cdot-(e_j_j+1*e_j+1_j+2))(1_j+2)} we have
	\begin{align*}
		(e_{i,i+1}*e_{j,j+1})(1,j+1)&=(e_{j,j+1}*e_{j+1,j+2})(2,j+2)=((e_{j,j+1}*e_{j+1,j+2})\cdot e_{j+2,j+3})(2,j+3)\\
		&=(e_{j,j+1}\cdot (e_{j+1,j+2}*e_{j+2,j+3}))(2,j+3),
	\end{align*}
	which is zero because $j\ne i+1=2$.
\end{proof}

The following corollary is based on the idea that will be used several times below: as soon as we have a statement for $*$, we can apply it to the product $\star$ from \cref{from-*-to-star}. 

\begin{cor}\label{e_ij*e_kl-in-span-e_1n_e_2n}
	For all $2\le i\le n-1$ and $1\le j\le n-1$ if $i+1\ne j$, then $e_{i,i+1}*e_{j,j+1}\in\spn_K\{e_{1n},e_{2n}\}$.
\end{cor}
\begin{proof}
	%Consider the involution $\phi$ of $UT_n(K)$ given by \cref{vf(e_kl)=e_(n-l+1_n-k+1)}.
%	\begin{align}\label{vf(e_kl)=e_(n-l+1_n-k+1)}
%		\phi(e_{kl})=e_{n-l+1,n-k+1}
%	\end{align}
%	for all $1\le k<l\le n$. 
	%Let $\star$ be the product on $UT_n(K)$ defined by \cref{a-star-b=phi(phi-inv(b)*phi-inv(a))}. 
	If $2\le i\le n-1$, $1\le j\le n-1$ and $i+1\ne j$, then $1\le n-j\le n-1$, $1\le n-i\le n-2$ and $n-j+1\ne n-i$. Hence, applying \cref{e_ij*e_kl-in-span-e_1_n-1_e_1n} to the product $\star$ from \cref{from-*-to-star} we obtain
	\begin{align*}
		e_{n-j,n-j+1}\star e_{n-i,n-i+1}\in\spn_K\{e_{1,n-1},e_{1n}\}.
	\end{align*}
 Then by \cref{a-star-b=phi(phi-inv(b)*phi-inv(a))} we have $\phi(e_{i,i+1}*e_{j,j+1})\in\spn_K\{e_{1,n-1},e_{1n}\}$. Applying $\phi\m=\phi$, we get the desired result.
\end{proof}

\begin{cor}\label{e_ii+1*e_jj+1-in-span-e_1n}
	For all $2\le i\le n-1$ and $1\le j\le n-2$ if $i+1\ne j$, then $e_{i,i+1}*e_{j,j+1}\in\spn_K\{e_{1n}\}$.
\end{cor}
\begin{proof}
	By \cref{e_ij*e_kl-in-span-e_1_n-1_e_1n,e_ij*e_kl-in-span-e_1n_e_2n} we have
	\begin{align*}
		e_{i,i+1}*e_{j,j+1}&\in\spn_K\{e_{1,n-1},e_{1n}\}\cap\spn_K\{e_{1n},e_{2n}\}=\spn_K\{e_{1n}\}.\qedhere
	\end{align*} 
\end{proof}

\begin{lem}\label{e_12*e_n-1_n-in-span-e_1_n-1_e_1n_e_2n}
	We have $e_{12}*e_{n-1,n}\in\spn_K\{e_{1,n-1},e_{1n},e_{2n}\}$.
\end{lem}
\begin{proof}
	The statement is trivial for $n=3$, so let $n>3$. Take $2<k<n$. Then $e_{k-1,k}*e_{12}\in\spn_K\{e_{1n},e_{2n}\}$ by \cref{e_ij*e_kl-in-span-e_1n_e_2n}. Therefore, $(e_{k-1,k}*e_{12})\cdot e_{n-1,n}=0$, whence $e_{k-1,k}\cdot (e_{12}*e_{n-1,n})=0$. Consequently, 
	\begin{align}\label{e_12*e_n-1_n-in-span-e_1l-oplus-span-e_2l}
		e_{12}*e_{n-1,n}\in\spn_K\{e_{st}: 1\le s<t\le n,\ s\in\{1,2\}\}.
	\end{align}
Applying \cref{e_12*e_n-1_n-in-span-e_1l-oplus-span-e_2l} to the product $\star$ from \cref{from-*-to-star}, we get
\begin{align}\label{e_12*e_n-1_n-in-span-e_k_n-1-oplus-span-e_kn}
	e_{12}*e_{n-1,n}\in\spn_K\{e_{st}: 1\le s<t\le n,\ t\in\{n-1,n\}\}.
\end{align}
It follows from \cref{e_12*e_n-1_n-in-span-e_1l-oplus-span-e_2l,e_12*e_n-1_n-in-span-e_k_n-1-oplus-span-e_kn} that $e_{12}*e_{n-1,n}\in\spn_K\{e_{1,n-1},e_{1n},e_{2,n-1},e_{2n}\}$. However, since $e_{n-1,n}\cdot e_{n-1,n}=0$, we have 
\begin{align*}
	e_{12}\cdot (e_{12}*e_{n-1,n})\cdot e_{n-1,n}=(e_{12}*e_{12})\cdot e_{n-1,n}\cdot e_{n-1,n}=0,
\end{align*}
whence $(e_{12}*e_{n-1,n})(2,n-1)=0$.
\end{proof}

\begin{lem}\label{e_12*e_ii+1(1_n-1)=e_ii+1*e_n-1n(2_n)}
	For all $1\le i\le n-1$ we have 
	$
		(e_{12}*e_{i,i+1})(1,n-1)=(e_{i,i+1}*e_{n-1,n})(2,n).
	$
\end{lem}
\begin{proof}
	Just evaluate $(e_{12}*e_{i,i+1})\cdot e_{n-1,n}=e_{12}\cdot (e_{i,i+1}*e_{n-1,n})$ at $(1,n)$.
\end{proof}

\subsubsection{The product $e_{ik}*e_{kj}$}\label{sec-(12)-e_ik*e_kj}

\begin{lem}\label{e_ik*e_kj-does-not-depend-on-k}
	Let $1\le i<j\le n$ with $j>i+1$. Then for all $i<k<j$ we have $e_{ik}*e_{kj}=e_{i,i+1}*e_{i+1,j}$.
%	\begin{align*}%\label{e_ik*e_kj=e_il*e_lj}
%		e_{ik}*e_{kj}=e_{i,i+1}*e_{i+1,j}.
%	\end{align*}
	Thus, $e_{ik}*e_{kj}$ does not depend on $k$.
\end{lem}
\begin{proof}
	It suffices to consider the case $k>i+1$. Then 
	\begin{align*}
		e_{ik}*e_{kj}&=(e_{i,i+1}\cdot e_{i+1,k})*e_{kj}=e_{i,i+1}*(e_{i+1,k}\cdot e_{kj})=e_{i,i+1}*e_{i+1,j}.\qedhere
	\end{align*}
\end{proof}

\begin{lem}\label{e_ik*e_kj-for-1-le-i<k<j<n}
	For all $1\le i<k<j<n$ we have $e_{ik}*e_{kj}\in\spn_K\{e_{ij},e_{1,n-1},e_{1n}\}$.
\end{lem}
\begin{proof}
	In view of \cref{e_ik*e_kj-does-not-depend-on-k} it suffices to take $k=i+1$. Let $1<l\le n$ with $l\ne i$. Then $e_{l-1,l}*e_{i,i+1}\in\spn_K\{e_{1,n-1},e_{1n}\}$ by \cref{e_ij*e_kl-in-span-e_1_n-1_e_1n}. It follows that $(e_{l-1,l}*e_{i,i+1})\cdot e_{i+1,j}=0$, whence $e_{l-1,l}\cdot(e_{i,i+1}*e_{i+1,j})=0$, i.e.
	\begin{align}\label{e_ik*e_kj-in-span-e_it-oplus-span-e_1t}
		e_{i,i+1}*e_{i+1,j}\in\spn_K\{e_{st}: 1\le s<t\le n,\ s\in\{1,i\}\}.
	\end{align}
	Now, for $1\le t<n-1$ with $t\ne j$ either $e_{i+1,j}*e_{t,t+1}=0$ (if $j>i+2$) by \cref{e_ij*e_kl=0-for-j-ne-k-and-j-i>1-or-l-k>1} or $e_{i+1,j}*e_{t,t+1}\in\spn_K\{e_{1,n-1},e_{1n}\}$ (if $j=i+2$) by \cref{e_ij*e_kl-in-span-e_1_n-1_e_1n}. In any case, $e_{i,i+1}\cdot (e_{i+1,j}*e_{t,t+1})=0$, whence $(e_{i,i+1}*e_{i+1,j})\cdot e_{t,t+1}=0$. Thus, $t\in\{j,n-1,n\}$ in \cref{e_ik*e_kj-in-span-e_it-oplus-span-e_1t}, i.e.
	\begin{align}\label{e_ik*e_kj-in-span-e_it-t-in-j-n-1-n}
		e_{i,i+1}*e_{i+1,j}\in\spn_K\{e_{st}:s\in\{1,i\}, t\in\{j,n-1,n\}\}.
	\end{align}
	If $i=1$, then we are done. Otherwise, 
	\begin{align*}
		e_{1i}*e_{i,i+1}\in\spn_K\{e_{1t}: t\in\{i+1,n-1,n\}\}
	\end{align*}
	by \cref{e_ik*e_kj-in-span-e_it-t-in-j-n-1-n}. Then
	\begin{align*}
		e_{1i}\cdot(e_{i,i+1}*e_{i+1,j})=(e_{1i}*e_{i,i+1})\cdot e_{i+1,j}\in\spn_K\{e_{1j}\},
	\end{align*}
	showing that one can remove $(s,t)=(i,n-1)$ and $(s,t)=(i,n)$ from \cref{e_ik*e_kj-in-span-e_it-t-in-j-n-1-n}. If $j=n-1$, then we are done. Otherwise,
	\begin{align*}
		e_{i+1,j}*e_{j,n-1}\in\spn_K\{e_{st}:s\in\{1,i+1\}, t\in\{n-1,n\}\}
	\end{align*}
	by \cref{e_ik*e_kj-in-span-e_it-t-in-j-n-1-n}. Then
	\begin{align*}
		(e_{i,i+1}*e_{i+1,j})\cdot e_{j,n-1}=e_{i,i+1}\cdot(e_{i+1,j}*e_{j,n-1})\in\spn_K\{e_{i,n-1},e_{in}\},
	\end{align*}
	showing that one can also remove $(s,t)=(1,j)$ from \cref{e_ik*e_kj-in-span-e_it-t-in-j-n-1-n}. 
\end{proof}

\begin{cor}\label{e_ik*e_kj-for-1<i<k<j-le-n}
	For all $1<i<k<j\le n$ we have $e_{ik}*e_{kj}\in\spn_K\{e_{ij},e_{1n},e_{2n}\}$.
\end{cor}
\begin{proof}
	This follows by applying \cref{e_ik*e_kj-for-1-le-i<k<j<n} to the product $\star$ from \cref{from-*-to-star}.
\end{proof}

\begin{cor}
	For all $1<i<k<j<n$ we have $e_{ik}*e_{kj}\in\spn_K\{e_{ij},e_{1n}\}$.
\end{cor}
\begin{proof}
	By \cref{e_ik*e_kj-for-1-le-i<k<j<n,e_ik*e_kj-for-1<i<k<j-le-n} we have
	\begin{align*}
		e_{ik}*e_{kj}&\in\spn_K\{e_{ij},e_{1,n-1},e_{1n}\}\cap\spn_K\{e_{ij},e_{1n},e_{2n}\}=\spn_K\{e_{ij},e_{1n}\}.\qedhere
	\end{align*}
\end{proof}

%We now specify the result of \cref{e_ik*e_kj-for-1-le-i<k<j<n}
\begin{lem}\label{e_1k*e_kj-in-span-e_1j-e_1n}
	For all $1<k<j\le n-1$ we have $e_{1k}*e_{kj}\in\spn_K\{e_{1j},e_{1n}\}$, unless $(k,j)=(2,3)$.
\end{lem}
\begin{proof}
	If $j=n-1$, then the result trivially follows from \cref{e_ik*e_kj-for-1-le-i<k<j<n}, so let $j<n-1$. In view of \cref{e_ik*e_kj-for-1-le-i<k<j<n} it suffices to show that $(e_{1k}*e_{kj})(1,n-1)=0$ for $(k,j)\ne (2,3)$. By \cref{e_ik*e_kj-does-not-depend-on-k} we have $e_{1k}*e_{kj}=e_{12}*e_{2j}$. If $j>3$, then $e_{2j}*e_{n-1,n}=0$ by \cref{e_ij*e_kl=0-for-j-ne-k-and-j-i>1-or-l-k>1}, whence
	\begin{align*}
		(e_{12}*e_{2j})(1,n-1)=((e_{12}*e_{2j})\cdot e_{n-1,n})(1,n)=(e_{12}\cdot (e_{2j}*e_{n-1,n}))(1,n)=0.
	\end{align*}
	If $j=3$, then $k=2$, and we are done.
\end{proof}

\begin{cor}
	For all $1\le i<k<n$ we have $e_{ik}*e_{kn}\in\spn_K\{e_{in},e_{1n}\}$, unless $(i,k)=(n-2,n-1)$.
\end{cor}
\begin{proof}
	This follows by applying \cref{e_1k*e_kj-in-span-e_1j-e_1n} to the product $\star$ from \cref{from-*-to-star}.
\end{proof}

\begin{lem}
	Let $n>3$. Then for all $1<k<n$ we have $e_{1k}*e_{kn}\in\spn_K\{e_{1n}\}$.
\end{lem}
\begin{proof}
	In view of \cref{e_ik*e_kj-does-not-depend-on-k} it suffices to prove the result for $k=2$. For all $1<l\le n$ we have $e_{l-1,l}*e_{12}\in\spn_K\{e_{1,n-1},e_{1n}\}$ by \cref{e_ij*e_kl-in-span-e_1_n-1_e_1n}. Then $(e_{l-1,l}*e_{12})\cdot e_{2n}=0$ because $n>3$, whence $e_{l-1,l}\cdot(e_{12}*e_{2n})=0$. Now, for all $1\le l<n$ we have $e_{2n}*e_{l,l+1}=0$ by \cref{e_ij*e_kl=0-for-j-ne-k-and-j-i>1-or-l-k>1}, because $n>3$. Then $e_{12}\cdot(e_{2n}*e_{l,l+1})=0$, whence $(e_{12}*e_{2n})\cdot e_{l,l+1}=0$. Since $\{e_{l,l+1}: 1\le l<n\}$ generate $UT_n(K)$, we have $e_{12}*e_{2n}\in\Ann(UT_n(K),\cdot)=\spn_K\{e_{1n}\}$.
\end{proof}

\begin{lem}
	Let $n>3$. Then for all $1\le i<k<j\le n$ with $(i,j)\ne (1,n)$ we have $(e_{ik}*e_{kj})(i,j)=(e_{12}*e_{23})(1,3)$.
\end{lem}
\begin{proof}
	It suffices to prove that $(e_{ik}*e_{kj})(i,j)=(e_{23}*e_{3n})(2,n)$, because
	\begin{align*}
		(e_{23}*e_{3n})(2,n)=(e_{12}\cdot(e_{23}*e_{3n}))(1,n)&=((e_{12}*e_{23})\cdot e_{3n})(1,n)
		=(e_{12}*e_{23})(1,3).
	\end{align*}
	
	Assume that $i>1$. Then $e_{ik}*e_{kj}\in\spn_K\{e_{ij},e_{1n},e_{2n}\}$ by \cref{e_ik*e_kj-for-1<i<k<j-le-n} and $e_{1i}*e_{ik}\in\spn_K\{e_{1k},e_{1,n-1},e_{1n}\}$ by \cref{e_ik*e_kj-for-1-le-i<k<j<n}. Consequently,
	\begin{align*}
		(e_{ik}*e_{kj})(i,j)=(e_{1i}\cdot(e_{ik}*e_{kj}))(1,j)=((e_{1i}*e_{ik})\cdot e_{kj})(1,j)=(e_{1i}*e_{ik})(1,k),
	\end{align*}
	the last being $(e_{12}*e_{2k})(1,k)$ by \cref{e_ik*e_kj-does-not-depend-on-k}. Now, $k<n$, so 
	\begin{align*}
		(e_{12}*e_{2k})(1,k)&=((e_{12}*e_{2k})\cdot e_{kn})(1,n)=(e_{12}\cdot (e_{2k}*e_{kn}))(1,n)
		=(e_{2k}*e_{kn})(2,n)=(e_{23}*e_{3n})(2,n),
	\end{align*}
	where the last equality is again by \cref{e_ik*e_kj-does-not-depend-on-k}. 
	
	Now let $i=1$, so that $j<n$. Then $e_{ik}*e_{kj}=e_{1k}*e_{kj}=e_{12}*e_{2j}$, the latter being by \cref{e_ik*e_kj-does-not-depend-on-k}. We have
	\begin{align*}
		(e_{12}*e_{2j})(1,j)=((e_{12}*e_{2j})\cdot e_{jn})(1,n)=(e_{12}\cdot(e_{2j}* e_{jn}))(1,n)=(e_{2j}*e_{jn})(2,n).
	\end{align*}
	The last again equals $(e_{23}*e_{3n})(2,n)$ by \cref{e_ik*e_kj-does-not-depend-on-k}.
\end{proof}

\begin{thrm}\label{(12)-matching-UT_n-full-descr}
	If $n>4$, then the $(12)$-matching structures on $UT_n(K)$ are exactly the associative products that are linear combinations of the structures:
	\begin{enumerate}
		\item\label{e_ii+1*e_jj+1=e_1n} $\M^{(12),1}_{i,j} (1<i\le n, 1\le j<n):$ $e_{i-1,i}*e_{j,j+1}=e_{1n}$; 
		
		\item\label{e_12*e_ii+1=e_1n-1_e_ii+1*e_n-1n=e_2n} $\M^{(12),2}_i (1\le i<n):$ 
			$e_{12}*e_{i,i+1}=e_{1,n-1}$, $e_{i,i+1}*e_{n-1,n}=e_{2n}$;
		
		\item\label{e_ik*e_kj=e_1n} $\M^{(12),3}_{i,j} (1\le i<j\le n, j>i+2):$ $e_{ik}*e_{kj}=e_{1n}$ for all $i<k<j$;

		\item\label{e_ik*e_kj=e_ij} $\M^{(12),4}:$ $e_{ik}*e_{kj}=e_{ij}$ for all $1\le i<k<j\le n$ with $(i,j)\ne(1,n)$.
	\end{enumerate}
	If $n=4$ (resp. $n=3$), then the $(12)$-matching structures on $UT_n(K)$ are exactly the associative products that are linear combinations of  \cref{e_ii+1*e_jj+1=e_1n,e_12*e_ii+1=e_1n-1_e_ii+1*e_n-1n=e_2n,e_ik*e_kj=e_1n} (resp. \cref{e_ii+1*e_jj+1=e_1n,e_12*e_ii+1=e_1n-1_e_ii+1*e_n-1n=e_2n}) above.
\end{thrm}
\begin{proof}
	Let $*$ be a $(12)$-matching structure on $UT_n(K)$. Then by the results of \cref{sec-(12)-e_ij*e_kl-j-ne-k} %\cref{e_ij*e_kl=0-for-j-ne-k-and-j-i>1-or-l-k>1,e_ij*e_kl-in-span-e_1_n-1_e_1n,e_ij*e_kl-in-span-e_1n_e_2n,e_ii+1*e_jj+1-in-span-e_1n,e_12*e_ii+1(1_n-1)=e_ii+1*e_n-1n(2_n)} 
	there exist $\{\af_{ij}: i\ne j\text{ unless }n=3\text{ and }i=j=2\}\sst K$ and $\{\bt_j\}_{j=1}^{n-1}\sst K$ such that
	\begin{align*}
		e_{i-1,i}*e_{j,j+1}&=\af_{ij}e_{1n},\ 3\le i\le n,\ 1\le j\le n-2,\ i\ne j,\\
		e_{12}*e_{j,j+1}&=\bt_je_{1,n-1}+\af_{2j}e_{1n},\ 1\le j\le n-2,\ j\ne 2,\\
		e_{j,j+1}*e_{n-1,n}&=\bt_je_{2n}+\af_{j+1,n-1}e_{1n},\ 2\le j\le n-1,\ j\ne n-2,\\
		e_{12}*e_{n-1,n}&=\bt_{n-1}e_{1,n-1}+\bt_1e_{2n}+\af_{2,n-1}e_{1n}.
	\end{align*}
	Moreover, by the results of \cref{sec-(12)-e_ik*e_kj} there exist $\{\af_{ii}\}_{i=2}^{n-1}\sst K$, $\bt_2,\bt_{n-2}\in K$, $\{\gm_{ij}: 2\le i+1<j<n\}\sst K$ and $\eta\in K$ such that
	\begin{align*}
		e_{ik}*e_{kj}&=\eta e_{ij}+\gm_{ij}e_{1n},\ j>i+2,\ (i,j)\ne (1,n),\\
		e_{i-1,i}*e_{i,i+1}&=\eta e_{i-1,i+1}+\af_{ii}e_{1n},\ 2<i<n-1,\\
		e_{1k}*e_{kn}&=\gm_{1n}e_{1n},\ n>3,\\
		e_{12}*e_{23}&=\begin{cases}
			\bt_2e_{12} + \bt_1e_{23} + \af_{22}e_{13}, & n=3,\\
			\bt_2e_{13}+\af_{22}e_{14}, & n=4,\\
			\eta e_{13}+\bt_2e_{1,n-1}+\af_{22}e_{1n}, & n>4,
		\end{cases}\\
		e_{n-2,n-1}*e_{n-1,n}&=
		\begin{cases}
			\bt_2e_{12} + \bt_1e_{23} + \af_{22}e_{13}, & n=3,\\
			\bt_2e_{24}+\af_{33}e_{14}, & n=4,\\
			\eta e_{n-2,n}+\bt_{n-2}e_{2n}+\af_{n-1,n-1}e_{1n}, & n>4.
		\end{cases}
	\end{align*}
	Thus, $*$ is the linear combination $\sum_{ij} \af_{ij}\M^{(12),1}_{i,j}+\sum_j\bt_j\M^{(12),2}_j+\sum_{i,j}\gm_{ij}\M^{(12),3}_{i,j}+\eta\M^{(12),4}$. Observe that, whenever $n=3$, the structures $\M^{(12),3}_{i,j}$ and $\M^{(12),4}$ do not exist due to the restrictions on the indices, and whenever $n=4$, we have $\M^{(12),4}=\M^{(12),2}_2$. The restriction $j>i+2$ in \cref{e_ik*e_kj=e_1n} was imposed to exclude the coincidence $\M^{(12),3}_{i,i+2}=\M^{(12),1}_{i+1,i+1}$.
	
	Conversely, each of the products \cref{e_ii+1*e_jj+1=e_1n,e_12*e_ii+1=e_1n-1_e_ii+1*e_n-1n=e_2n,e_ik*e_kj=e_1n,e_ik*e_kj=e_ij} is a $(12)$-matching structure on $UT_n(K)$ as seen in \cref{sec-(12)-matching-classes}. So, any associative linear combination of them is also a $(12)$-matching structure on $UT_n(K)$.
\end{proof}

\begin{rem}
	There are $(n-1)^2$ products of the form \cref{e_ii+1*e_jj+1=e_1n}, $n-1$ products of the form \cref{e_12*e_ii+1=e_1n-1_e_ii+1*e_n-1n=e_2n} and $\frac 12(n-2)(n-3)$ products of the form \cref{e_ik*e_kj=e_1n}. In total, there are $\frac n2(3n-7)+3$ linearly independent products of the form \cref{e_ii+1*e_jl=e_1l,e_ki*e_jj+1=e_kn,e_ki*e_jl=e_kl} and $1$ product of the form \cref{e_ik*e_kj=e_ij} (which is linearly independent with \cref{e_ii+1*e_jl=e_1l,e_ki*e_jj+1=e_kn,e_ki*e_jl=e_kl} only for $n>4$).
\end{rem}

\subsection{The complete classification of $(12)$-matching structures on $UT_3(K)$}\label{sec-(12)-match-n=3}

We are now ready to give a complete classification of the $(12)$-matching structures on $UT_3(K)$ up to an isomorphism. Specifying \cref{(12)-matching-UT_n-full-descr} to the case $n=3$, we get the following.

\begin{cor}\label{structures-M^(12)1_ij-n=3}
	The $(12)$-matching structures on $UT_3(K)$ are exactly the associative products that are linear combinations of the structures:
	\begin{multicols}{2}
	\begin{enumerate}
		\item\label{M^(12)1_21-n=3} $\M^{(12),1}_{2,1}:$ $e_{12}*e_{12}=e_{13}$;
		\item\label{M^(12)1_22-n=3} $\M^{(12),1}_{2,2}:$ $e_{12}*e_{23}=e_{13}$;
		\item\label{M^(12)1_31-n=3} $\M^{(12),1}_{3,1}:$ $e_{23}*e_{12}=e_{13}$;
		\item\label{M^(12)1_32-n=3} $\M^{(12),1}_{3,2}:$ $e_{23}*e_{23}=e_{13}$;
		\item\label{M^(12)2_1-n=3} $\M^{(12),2}_1:$ 
		$e_{12}*e_{12}=e_{12}$, $e_{12}*e_{23}=e_{23}$;
		\item\label{M^(12)2_2-n=3} $\M^{(12),2}_2:$ 
		$e_{12}*e_{23}=e_{12}$, $e_{23}*e_{23}=e_{23}$.
	\end{enumerate}
	\end{multicols}
	i.e. the associative products of the form
	\begin{align}\label{(12)-matching-UT_3(K)-general-form}
		\begin{array}{rclrcl}
			e_{12}*e_{12}&=&\lb_5 e_{12}+\lb_1 e_{13},& e_{23}*e_{12}&=&\lb_3 e_{13},\\
			e_{12}*e_{23}&=&\lb_6 e_{12}+\lb_2 e_{13}+\lb_5 e_{23},& e_{23}*e_{23}&=&\lb_4 e_{13}+\lb_6 e_{23}\\
		\end{array}
	\end{align}
	for some $\{\lb_i\}_{i=1}^6\sst K$.
\end{cor}

As above, we omit the proofs of the following two lemmas, since they are just a  straightforward calculation.
\begin{lem}\label{associators-UT_3-(12)-match}
	We have the following non-trivial associators of the basis elements $\{e_{12},e_{23},e_{13}\}$ of $UT_3(K)$ under \cref{(12)-matching-UT_3(K)-general-form}:
	\begin{align}
		[e_{12}, e_{12}, e_{23}]&= -\lb_5\lb_6e_{12} - \lb_1\lb_6e_{13},\label{[e_12_e_12_e_23]-(12)}\\
		[e_{12}, e_{23}, e_{12}]&= \lb_5\lb_6e_{12} + (\lb_1\lb_6+\lb_3\lb_5)e_{13},\label{[e_12_e_23_e_12]-(12)}\\
		[e_{12}, e_{23}, e_{23}]&= \lb_4\lb_5e_{13}+\lb_5\lb_6e_{23},\label{[e_12_e_23_e_23]-(12)}\\
		[e_{23}, e_{12}, e_{12}]&= -\lb_3\lb_5e_{13},\notag\\
		[e_{23}, e_{12}, e_{23}]&= -(\lb_3\lb_6 + \lb_4\lb_5)e_{13}-\lb_5\lb_6e_{23},\notag\\
		[e_{23}, e_{23}, e_{12}]&= \lb_3\lb_6e_{13}.\label{[e_23_e_23_e_12]-(12)}
	\end{align}
\end{lem}

\begin{lem}\label{(12)-matching-UT_3(K)-new-base}
	Let $\phi\in\Aut(UT_3(K))$ as in \cref{aut-UT_3-on-e_12_e_23_e_13} and $*$ be given by \cref{(12)-matching-UT_3(K)-general-form}. Then the corresponding product $\star$ on $UT_3(K)$ defined by \cref{a-star-b=phi(phi-inv(a)*phi-inv(b))} is as follows:
	\begin{align*}
		e_{12} \star e_{12} &= a_{11}^{-1}\lb_5e_{12} + a_{11}^{-2}(a_{11}a_{22}\lb_1 + a_{31}\lb_5)e_{13},\\
		e_{12} \star e_{23} &= a_{22}^{-1}\lb_6e_{12} + a_{11}^{-1}\lb_5e_{23} + a_{11}^{-1}a_{22}^{-1}(a_{11}a_{22}\lb_2 + a_{32}\lb_5 + a_{31}\lb_6)e_{13},\\
		e_{23} \star e_{12} &= \lb_3e_{13},\\
		e_{23} \star e_{23} &= a_{22}^{-1}\lb_6e_{23} + a_{22}^{-2}(a_{11}a_{22}\lb_4 + a_{32}\lb_6)e_{13}.
	\end{align*} 
\end{lem}

\begin{prop}\label{(12)-matching-UT_3-full-descr}
	Any $(12)$-matching structure on $UT_3(K)$ is isomorphic to exactly one of the following structures:
	\begin{enumerate}
		\item\label{star-for-lb_5-ne-0-(12)} $e_{12}\star e_{12} = e_{12}$, $e_{12} \star e_{23} = e_{23}$;
		\item\label{star-for-lb_5=0-lb_6-ne-0-(12)} $e_{12} \star e_{23} = e_{12}$, $e_{23} \star e_{23} = e_{23}$;
		\item\label{star-for-lb_5=0-lb_6=0-lb_4-ne-0-(12)} $e_{12}\star e_{12} = \af e_{13},\ e_{12}\star e_{23} = \bt e_{13},\ e_{23}\star e_{12} = \gm e_{13},\ e_{23}\star e_{23} = e_{13}$;
		\item\label{star-for-lb_5=0-lb_6=0-lb_4=0-lb_1-ne-0-(12)} $e_{12}\star e_{12} = e_{13},\ e_{12} \star e_{23} = \af e_{13},\ e_{23} \star e_{12} = \bt e_{13}$;
		\item\label{star-for-lb_5=0-lb_6=0-lb_1=0-lb_4=0-(12)} $e_{12}\star e_{23}=\af e_{13}$, $e_{23}\star e_{12}=\bt e_{13}$.
	\end{enumerate}
\end{prop}
\begin{proof}
	Let $*$ be an $\id$-matching structure on $UT_3(K)$. By \cref{structures-M^(12)1_ij-n=3} the product $*$ is of the form \cref{(12)-matching-UT_3(K)-general-form}. 
	
	\textit{Case 1.} $\lb_5\ne 0$. It follows from \cref{[e_12_e_12_e_23]-(12)} that $\lb_6=0$. Furthermore, by \cref{[e_12_e_23_e_12]-(12)} we have $\lb_3=0$ and by \cref{[e_12_e_23_e_23]-(12)} we have $\lb_4=0$. Let $\phi\in\Aut(UT_3(K))$ given by \cref{aut-UT_3-on-e_12_e_23_e_13}. Then by \cref{(12)-matching-UT_3(K)-new-base} the automorphism $\phi$ defines the following isomorphic structure
	\begin{align*}
		e_{12} \star e_{12} = a_{11}^{-1}\lb_5e_{12} + a_{11}^{-2}(a_{11}a_{22}\lb_1 + a_{31}\lb_5)e_{13},\ 
		e_{12} \star e_{23} = a_{11}^{-1}\lb_5e_{23} + a_{11}^{-1}a_{22}^{-1}(a_{11}a_{22}\lb_2 + a_{32}\lb_5)e_{13}.
	\end{align*}
	Choosing $a_{11}=\lb_5$, $a_{22}=1$, $a_{31} =-\lb_1$ and $a_{32} =-\lb_2$, we get \cref{star-for-lb_5-ne-0-(12)}.

	\textit{Case 2.} $\lb_5=0$ and $\lb_6\ne 0$. Then it follows from \cref{[e_12_e_12_e_23]-(12)} that $\lb_1=0$, and \cref{[e_23_e_23_e_12]-(12)} implies $\lb_3=0$. By \cref{(12)-matching-UT_3(K)-new-base} the automorphism $\phi$ defines the following isomorphic structure
	\begin{align*}
		e_{12} \star e_{23} = a_{22}^{-1} \lb_6 e_{12} + a_{11}^{-1} a_{22}^{-1}(a_{11} a_{22} \lb_2 +  a_{31} \lb_6) e_{13},\
		e_{23} \star e_{23} = a_{22}^{-1} \lb_6 e_{23} + a_{22}^{-2}(a_{11}a_{22} \lb_4 + a_{32} \lb_6) e_{13}.
	\end{align*}
	Choosing $a_{11}=1$, $a_{22}=\lb_6$, $a_{31} = -\lb_2$ and $a_{32} = -\lb_4$, we get \cref{star-for-lb_5=0-lb_6-ne-0-(12)}.
	
	\textit{Case 3.} $\lb_5=\lb_6=0$. Then $*$ is associative. By \cref{(12)-matching-UT_3(K)-new-base} the automorphism $\phi$ defines the following isomorphic structure
	\begin{align*}
		e_{12}\star e_{12}=a_{22}a_{11}^{-1}\lb_1e_{13},\
		e_{12}\star e_{23}=\lb_2e_{13},\
		e_{23}\star e_{12}=\lb_3e_{13},\
		e_{23}\star e_{23}=a_{11}a_{22}^{-1}\lb_4e_{13}.
	\end{align*}
	If $\lb_4\ne 0$, we get the family \cref{star-for-lb_5=0-lb_6=0-lb_4-ne-0-(12)}. If $\lb_4=0$ and $\lb_1\ne 0$, we get the family \cref{star-for-lb_5=0-lb_6=0-lb_4=0-lb_1-ne-0-(12)}. Finally, if $\lb_1=\lb_4=0$, we get the family \cref{star-for-lb_5=0-lb_6=0-lb_1=0-lb_4=0-(12)}.
	
	Conversely, each product of the list \cref{star-for-lb_5-ne-0-(12),star-for-lb_5=0-lb_6-ne-0-(12),star-for-lb_5=0-lb_6=0-lb_4-ne-0-(12),star-for-lb_5=0-lb_6=0-lb_4=0-lb_1-ne-0-(12),star-for-lb_5=0-lb_6=0-lb_1=0-lb_4=0-(12)} is a particular case of \cref{(12)-matching-UT_3(K)-general-form}, and it is associative by \cref{associators-UT_3-(12)-match}, so it is a $(12)$-matching structure on $UT_3(K)$ by \cref{structures-M^(12)1_ij-n=3}. Moreover, it is directly seen by \cref{(12)-matching-UT_3(K)-new-base} that all the listed structures are pairwise non-isomorphic.
\end{proof}

	\section{Interchangeable structures on $UT_n(K)$}\label{sec-interchangeable}

	\subsection{Some general lemmas}

	\begin{lem}\label{a-cdot-b=0=>a*b-in-Ann(A)}
		Let $(A,\cdot)$ be a (not necessarily associative) algebra and $*$ a bilinear product on $A$ interchangeable with $\cdot$. Then for all $a,b\in A$ if $a\cdot b=0$, then $a*b\in\Ann(A,\cdot)$.
	\end{lem}
	\begin{proof}
		Let $a\cdot b=0$. Then for all $c\in A$ we have $(a*b)\cdot c=(a\cdot b)*c=0$ and $c\cdot (a*b)=c*(a\cdot b)=0$.
	\end{proof}
	
	\begin{lem}\label{(a*b)-cdot-(c*d)-several-forms}
		Let $(A,\cdot)$ be an associative algebra and $*$ a bilinear product on $A$ interchangeable with $\cdot$. Then
		\begin{align*}
			(a*b)\cdot (c\cdot d)=(a\cdot b)\cdot (c*d)=(a\cdot b\cdot c)*d=a*(b\cdot c\cdot d)=(a\cdot b)*(c\cdot d)
		\end{align*}
		for all $a,b,c,d\in A$.
	\end{lem}
	\begin{proof}
		For all $a,b,c,d\in A$ we have 
		\begin{align*}
			(a\cdot b)*(c\cdot d)&=(a*b)\cdot (c\cdot d)=((a*b)\cdot c)\cdot d=((a\cdot b)*c)\cdot d=((a\cdot b)\cdot c)*d,\\
			(a\cdot b)*(c\cdot d)&=(a\cdot b)\cdot (c*d)=a\cdot(b\cdot (c*d))=a\cdot(b*(c\cdot d))=a*(b\cdot (c\cdot d)).\qedhere
		\end{align*}
	\end{proof}
	
	\subsection{Some classes of products on $UT_n(K)$ that are interchangeable with $\cdot$}\label{sec-interchangeable-classes}
	
	We begin by observing that $\M^{\id,4}_{i,j}$ from \cref{structures-M^id1_ij} is interchangeable with $\cdot$ for all $1<i\le n$ and $1\le j<n$ by \cref{M^id4_ij-totally-comp}.
	
%	The following is explained the same way as \cref{ann-valued-structure-(12)}.
%	\begin{rem}
%		The bilinear product \cref{e_i-1i*e_jj+1=e_1n} on $UT_n(K)$ is interchangeable with $\cdot$.
%	\end{rem}
	
	\begin{lem}\label{e_12*e_jj+1=e_1n-and-e_jj+1*e_j+1j+2=e_2n}
		For each fixed $1\le j\le n-2$ the bilinear product
		\begin{align*}
			e_{12}*e_{j,j+2}=e_{1n},\ e_{j,j+1}*e_{j+1,j+2}=e_{2n}
		\end{align*}
		is interchangeable with $\cdot$.
	\end{lem}
	\begin{proof}
		Since $e_{12}, e_{j,j+1}\in\spn_K\{e_{i,i+1}: 1\le i<n\}$, which is a complement of $UT_n(K)^2$, then $(a\cdot b)*c=0$ for all $a,b,c\in UT_n(K)$. Furthermore, $e_{1n}\cdot UT_n(K)=e_{2n}\cdot UT_n(K)=\{0\}$, whence $(a*b)\cdot c=0$ for all $a,b,c\in UT_n(K)$.
		
		Now, 
		\begin{align*}
			e_{xy}*(e_{zw}\cdot e_{uv})\ne 0\iff e_{xy}\cdot(e_{zw}*e_{uv})\ne 0&\iff (x,y)=(1,2),\ (z,w)=(j,j+1)
			\text{ and }(u,v)=(j+1,j+2)
		\end{align*}
		for some $1\le j\le n-2$, in which case 
		\begin{align*}
			e_{xy}*(e_{zw}\cdot e_{uv})&=e_{12}*(e_{j,j+1}\cdot e_{j+1,j+2})=e_{12}*e_{j,j+2}=e_{1n},\\
			e_{xy}\cdot(e_{zw}*e_{uv})&=e_{12}\cdot(e_{j,j+1}*e_{j+1,j+2})=e_{12}\cdot e_{2n}=e_{1n}.
		\end{align*}
		Thus, $a*(b\cdot c)=a\cdot(b*c)$ for all $a,b,c\in UT_n(K)$.
	\end{proof}
	
	\begin{cor}
		For each fixed $1\le j\le n-2$ the bilinear product
		\begin{align*}
			e_{j,j+2}*e_{n-1,n}=e_{1n},\ e_{j,j+1}*e_{j+1,j+2}=e_{1,n-1}
		\end{align*}
		is interchangeable with $\cdot$.
	\end{cor}
	\begin{proof}
		This follows by applying \cref{e_12*e_jj+1=e_1n-and-e_jj+1*e_j+1j+2=e_2n} to the product $\star$ from \cref{from-*-to-star}.
	\end{proof}
	
	\subsection{The decomposition of an interchangeable structure on $UT_n(K)$}
	
	Let $*$ be a (not necessarily associative) bilinear product on $UT_n(K)$ that is interchangeable with $\cdot$.
	
\subsubsection{The product $e_{ij}*e_{kl}$ with $j\ne k$}\label{sec-e_j*e_kl-inter}

\begin{lem}\label{e_ik*e_kl-in-spn-e_1n}
	For all $1\le i<j\le n$ and $1\le k<l\le n$ if $j\ne k$, then $e_{ij}*e_{kl}\in\spn_K\{e_{1n}\}$.
\end{lem}
\begin{proof}
	This is a consequence of \cref{a-cdot-b=0=>a*b-in-Ann(A)}, because $e_{ij}\cdot e_{kl}=0$ for $j\ne k$ and $\Ann(UT_n(K),\cdot)=\spn_K\{e_{1n}\}$.
\end{proof}

\begin{lem}\label{e_ij*e_kl=0-for-j-i+l-k>3}
	For all $1\le i<j\le n$ and $1\le k<l\le n$ if $j\ne k$ and $j-i+l-k>3$, then $e_{ij}*e_{kl}=0$.
\end{lem}
\begin{proof}
	The proof is based on \cref{(a*b)-cdot-(c*d)-several-forms}.
	
	\textit{Case 1.} $j-i>2$. Then $i<i+1<i+2<j$, so that 
	\begin{align*}
		e_{ij}*e_{kl}=(e_{i,i+1}\cdot e_{i+1,i+2}\cdot e_{i+2,j})*e_{kl}=(e_{i,i+1}*e_{i+1,i+2})\cdot (e_{i+2,j}\cdot e_{kl})
		=(e_{i,i+1}*e_{i+1,i+2})\cdot 0=0.
	\end{align*}
	
	\textit{Case 2.} $l-k>2$. Then $k<k+1<k+2<l$, so that
	\begin{align*}
		e_{ij}*e_{kl}=e_{ij}*(e_{k,k+1}\cdot e_{k+1,k+2}\cdot e_{k+2,l})=(e_{ij}\cdot e_{k,k+1})\cdot (e_{k+1,k+2}*e_{k+2,l})
		=0\cdot (e_{k+1,k+2}*e_{k+2,l})=0.
	\end{align*}
	
	\textit{Case 3.} $j-i=l-k=2$. Then $i<i+1<i+2=j$ and $k<k+1<k+2=l$, so that
	\begin{align*}
		e_{ij}*e_{kl}&=(e_{i,i+1}\cdot e_{i+1,j})*(e_{k,k+1}\cdot e_{k+1,l})=e_{i,i+1}*(e_{i+1,j}\cdot e_{k,k+1}\cdot e_{k+1,l})
		=e_{i,i+1}*0=0.\qedhere
	\end{align*}
\end{proof}

\begin{lem}\label{e_ii+1*e_jj+2=0}
	Let $1<i\le n-1$ and $1\le j\le n-2$ with $j\ne i+1$. Then $e_{i,i+1}*e_{j,j+2}=0$. 
\end{lem}
\begin{proof}
	We have 
	\begin{align*}
		e_{i,i+1}*e_{j,j+2}=e_{i,i+1}*(e_{j,j+1}\cdot e_{j+1,j+2})=e_{i,i+1}\cdot(e_{j,j+1}*e_{j+1,j+2}),
	\end{align*}
	whence $(e_{i,i+1}*e_{j,j+2})(1,n)=0$ because $i>1$. Since $e_{i,i+1}*e_{j,j+2}\in\spn_K\{e_{1n}\}$ by \cref{e_ik*e_kl-in-spn-e_1n}, we are done.
\end{proof}

\begin{cor}
	Let $1\le i\le n-2$ and $1\le j<n-1$ with $j\ne i+2$. Then $e_{i,i+2}*e_{j,j+1}=0$. 
\end{cor}
\begin{proof}
	This follows by applying \cref{e_ii+1*e_jj+2=0} to the product $\star$ from \cref{from-*-to-star}.
\end{proof}

\begin{lem}
	For all $1\le j\le n-2$ we have 
	\begin{align*}
		(e_{12}*e_{j,j+2})(1,n)&=(e_{j,j+1}*e_{j+1,j+2})(2,n),\\
		(e_{j,j+2}*e_{n-1,n})(1,n)&=(e_{j,j+1}*e_{j+1,j+2})(1,n-1).
	\end{align*}
\end{lem}
\begin{proof}
	Just evaluate the following products at $(1,n)$:
	\begin{align*}
		e_{12}*e_{j,j+2}&=e_{12}*(e_{j,j+1}\cdot e_{j+1,j+2})
		=e_{12}\cdot(e_{j,j+1}*e_{j+1,j+2}),\\
		e_{j,j+2}*e_{n-1,n}&=(e_{j,j+1}\cdot e_{j+1,j+2})*e_{n-1,n}
		=(e_{j,j+1}*e_{j+1,j+2})\cdot e_{n-1,n}.\qedhere
	\end{align*}
%	\begin{align*}
%		(e_{12}*e_{j,j+2})(1,n)&=(e_{12}*(e_{j,j+1}\cdot e_{j+1,j+2}))(1,n)\\
%		&=(e_{12}\cdot(e_{j,j+1}*e_{j+1,j+2}))(1,n)\\
%		&=(e_{j,j+1}*e_{j+1,j+2})(2,n),\\
%		(e_{j,j+2}*e_{n-1,n})(1,n)&=((e_{j,j+1}\cdot e_{j+1,j+2})*e_{n-1,n})(1,n)\\
%		&=((e_{j,j+1}*e_{j+1,j+2})\cdot e_{n-1,n})(1,n)\\
%		&=(e_{j,j+1}*e_{j+1,j+2})(1,n-1).
%	\end{align*}
\end{proof}

\subsubsection{The product $e_{ik}*e_{kj}$}\label{sec-e_ik*e_kj-inter}

\begin{lem}\label{e_ik*e_kj-in-span-e_ij-e_1j}
	Let $1\le i<k<j\le n$. If $k>i+1$, then $e_{ik}*e_{kj}\in\spn_K\{e_{ij},e_{1j}\}$.
\end{lem}
\begin{proof}
	We have 
	\begin{align}\label{e_ik*e_kj-in-span-e_kj}
		e_{ik}*e_{kj}=(e_{i,i+1}\cdot e_{i+1,k})*e_{kj}
		=(e_{i,i+1}*e_{i+1,k})\cdot e_{kj}\in\spn_K\{e_{lj}: 1\le l<j\le n\}.
	\end{align}
	Now, $j-i>2$, so for all $1<t<j$ with $t\ne i$ we have $e_{t-1,t}*e_{ij}=0$ by \cref{e_ij*e_kl=0-for-j-i+l-k>3}. Hence,
	\begin{align*}
		(e_{ik}*e_{kj})(t,j)=(e_{t-1,t}\cdot(e_{ik}*e_{kj}))(t-1,j)=(e_{t-1,t}*(e_{ik}\cdot e_{kj}))(t-1,j)
		=(e_{t-1,t}*e_{ij})(t-1,j)=0.
	\end{align*}
	Thus, in fact $l\in\{1,i\}$ in \cref{e_ik*e_kj-in-span-e_kj}.
\end{proof}

\begin{cor}\label{e_ik*e_kj-in-span-e_ij-e_in}
	Let $1\le i<k<j\le n$. If $j>k+1$, then $e_{ik}*e_{kj}\in\spn_K\{e_{ij},e_{in}\}$.
\end{cor}
\begin{proof}
	This follows by applying \cref{e_ik*e_kj-in-span-e_ij-e_1j} to the product $\star$ from \cref{from-*-to-star}.
\end{proof}

\begin{cor}\label{e_ik*e_kj-in-span-e_ij}
	Let $1\le i<k<j\le n$. If $k>i+1$ and $j>k+1$, then $e_{ik}*e_{kj}\in\spn_K\{e_{ij}\}$.
\end{cor}
\begin{proof}
	By \cref{e_ik*e_kj-in-span-e_ij-e_1j,e_ik*e_kj-in-span-e_ij-e_in} we have
	$
		e_{ik}*e_{kj}\in\spn_K\{e_{ij},e_{1j}\}\cap\spn_K\{e_{ij},e_{in}\}=\spn_K\{e_{ij}\}.
	$
\end{proof}

\begin{lem}\label{e_ik*e_kk+1-in-span-e_ik+1}
	Let $1\le i<k<n-1$. If $k>i+1$, then $e_{ik}*e_{k,k+1}\in\spn_K\{e_{i,k+1}\}$.
\end{lem}
\begin{proof}
	If $i=1$, then the statement trivially follows from \cref{e_ik*e_kj-in-span-e_ij-e_1j}. Let $i>1$. In view of \cref{e_ik*e_kj-in-span-e_ij-e_1j} it suffices to prove that $(e_{ik}*e_{k,k+1})(1,k+1)=0$. We have 
	\begin{align*}
		(e_{ik}*e_{k,k+1})(1,k+1)&=((e_{i,i+1}\cdot e_{i+1,k})*e_{k,k+1})(1,k+1)
		=((e_{i,i+1}*e_{i+1,k})\cdot e_{k,k+1})(1,k+1)\\
		&=(e_{i,i+1}*e_{i+1,k})(1,k)=((e_{i,i+1}*e_{i+1,k})\cdot e_{kn})(1,n)\\
		&=((e_{i,i+1}\cdot e_{i+1,k})*e_{kn})(1,n)=(e_{ik}*e_{kn})(1,n),
	\end{align*}
	which is zero by \cref{e_ik*e_kj-in-span-e_ij}, because $k>i+1$ and $n>k+1$ and $i>1$. 
\end{proof}

\begin{cor}\label{e_ii+1*e_i+1j-in-span-e_ij}
	Let $2<k<j\le n$. If $j>k+1$, then $e_{k-1,k}*e_{kj}\in\spn_K\{e_{k-1,j}\}$.
\end{cor}
\begin{proof}
	This follows by applying \cref{e_ik*e_kk+1-in-span-e_ik+1} to the product $\star$ from \cref{from-*-to-star}.
\end{proof}

\begin{lem}\label{e_in-1*e_n-1n-in-span-e_in}
	Let $1\le i<n-3$. Then $e_{i,n-1}*e_{n-1,n}\in\spn_K\{e_{in}\}$.
\end{lem}
\begin{proof}
	The statement is trivial for $i=1$ in view of \cref{e_ik*e_kj-in-span-e_ij-e_1j}. Let $i>1$. By \cref{e_ik*e_kj-in-span-e_ij-e_1j} it suffices to prove that $(e_{i,n-1}*e_{n-1,n})(1,n)=0$. We have 
	\begin{align*}
		(e_{i,n-1}*e_{n-1,n})(1,n)=((e_{i,i+1}\cdot e_{i+1,n-1})*e_{n-1,n})(1,n)
		&=((e_{i,i+1}*e_{i+1,n-1})\cdot e_{n-1,n})(1,n)\\
		&=(e_{i,i+1}*e_{i+1,n-1})(1,n-1),
	\end{align*}
	which is zero by \cref{e_ii+1*e_i+1j-in-span-e_ij} because $n-1>i+2$ and $i>1$.
\end{proof}

\begin{cor}\label{e_12*e_2j-in-span-e_in}
	Let $4<j\le n$. Then $e_{12}*e_{2j}\in\spn_K\{e_{1j}\}$.
\end{cor}
\begin{proof}
	This follows by applying \cref{e_ik*e_kk+1-in-span-e_ik+1}  to the product $\star$ from \cref{from-*-to-star}.
\end{proof}

\begin{lem}
	Let $1\le i\le n-2$. Then
	$
		e_{i,i+1}*e_{i+1,i+2}\in\spn_K\{e_{i,i+2},e_{1,n-1},e_{1n},e_{2n}\}.
	$
\end{lem}
\begin{proof}
	For all $2<s<t$ with $s\ne i$ we have $e_{s-1,s}*e_{i,i+2}\in\spn_K\{e_{1n}\}$ by \cref{e_ik*e_kl-in-spn-e_1n}. Hence,
	\begin{align*}
		(e_{i,i+1}*e_{i+1,i+2})(s,t)=(e_{s-1,s}\cdot(e_{i,i+1}*e_{i+1,i+2}))(s-1,t)
		&=(e_{s-1,s}*(e_{i,i+1}\cdot e_{i+1,i+2}))(s-1,t)\\
		&=(e_{s-1,s}*e_{i,i+2})(s-1,t)=0.
	\end{align*} 
	Similarly, for all $1\le s<t<n-1$ with $t\ne i+2$ we have $e_{i,i+2}*e_{t,t+1}\in\spn_K\{e_{1n}\}$ by \cref{e_ik*e_kl-in-spn-e_1n}. Hence,
	\begin{align*}
		(e_{i,i+1}*e_{i+1,i+2})(s,t)=((e_{i,i+1}*e_{i+1,i+2})\cdot e_{t,t+1})(s,t+1)
		&=((e_{i,i+1}\cdot e_{i+1,i+2})* e_{t,t+1})(s,t+1)\\
		&=(e_{i,i+2}*e_{t,t+1})(s,t+1)=0.
	\end{align*}
	Thus,
	\begin{align}\label{e_ii+1*e_i+1i+2-in-span-e_st}
		e_{i,i+1}*e_{i+1,i+2}\in\spn_K\{e_{st}: s\in\{1,2,i\},t\in\{i+2,n-1,n\}\}.
	\end{align}
	
	\textit{Case 1.} $i>2$. Then 
	\begin{align*}
		e_{1i}\cdot(e_{i,i+1}*e_{i+1,i+2})=e_{1i}*(e_{i,i+1}\cdot e_{i+1,i+2})=e_{1i}*e_{i,i+2}\in\spn_K\{e_{1,i+2}\}
	\end{align*}
	by \cref{e_ik*e_kj-in-span-e_ij}. So, one can remove $(s,t)=(i,n-1)$ and $(s,t)=(i,n)$ from \cref{e_ii+1*e_i+1i+2-in-span-e_st}. Furthermore,
	\begin{align*}
		e_{12}\cdot(e_{i,i+1}*e_{i+1,i+2})=e_{12}*(e_{i,i+1}\cdot e_{i+1,i+2})=e_{12}*e_{i,i+2}\in\spn_K\{e_{1n}\}
	\end{align*}
	by \cref{e_ik*e_kl-in-spn-e_1n}. Hence, $(s,t)=(2,i+2)$ and $(s,t)=(2,n-1)$ can also be removed from \cref{e_ii+1*e_i+1i+2-in-span-e_st}. If $i+2\in\{n-1,n\}$, we are done. Otherwise $i+2<n-1$, so we have 
	\begin{align*}
		(e_{i,i+1}*e_{i+1,i+2})\cdot e_{i+2,n}=(e_{i,i+1}\cdot e_{i+1,i+2})*e_{i+2,n}=e_{i,i+2}*e_{i+2,n}\in\spn_K\{e_{in}\}
	\end{align*}
	by \cref{e_ik*e_kj-in-span-e_ij}, and we can finally remove $(s,t)=(1,i+2)$ from \cref{e_ii+1*e_i+1i+2-in-span-e_st}.
	
	\textit{Case 2.} $i=2$. Then 
	\begin{align}\label{e_i_i+1*e_i+1_i+2-in-span-i=2}
		e_{i,i+1}*e_{i+1,i+2}\in\spn_K\{e_{st}: s\in\{1,2\},t\in\{4,n-1,n\}\}
	\end{align}
	by \cref{e_ii+1*e_i+1i+2-in-span-e_st}. To complete the proof, we must discard $(s,t)=(2,n-1)$ and $(s,t)=(1,4)$ from \cref{e_i_i+1*e_i+1_i+2-in-span-i=2}. We have 
	\begin{align*}
		e_{12}\cdot(e_{i,i+1}*e_{i+1,i+2})=e_{12}\cdot(e_{23}*e_{34})
		=e_{12}*(e_{23}\cdot e_{34})
		=e_{12}*e_{24}\in\spn_K\{e_{14},e_{1n}\}
	\end{align*}
	by \cref{e_ik*e_kj-in-span-e_ij-e_in}, so $(s,t)=(2,n-1)$ can be discarded from \cref{e_i_i+1*e_i+1_i+2-in-span-i=2}. If $4\in\{n-1,n\}$, we are done. Otherwise, $4<n-1$, so, as in Case 1, we have 
	\begin{align*}
		(e_{i,i+1}*e_{i+1,i+2})\cdot e_{4n}=(e_{23}*e_{34})\cdot e_{4n}=(e_{23}\cdot e_{34})*e_{4n}
		=e_{24}*e_{4n}\in\spn_K\{e_{2n}\}
	\end{align*}
	 by \cref{e_ik*e_kj-in-span-e_ij}, and we can also discard $(s,t)=(1,4)$ from \cref{e_i_i+1*e_i+1_i+2-in-span-i=2}.
	
	\textit{Case 3.} $i=1$. Then 
	\begin{align}\label{e_i_i+1*e_i+1_i+2-in-span-i=1}
		e_{i,i+1}*e_{i+1,i+2}\in\spn_K\{e_{st}: s\in\{1,2\},t\in\{3,n-1,n\}\}
	\end{align}
	by\cref{e_ii+1*e_i+1i+2-in-span-e_st}. To complete the proof, we must discard $(s,t)=(2,n-1)$ and $(s,t)=(2,3)$ from \cref{e_i_i+1*e_i+1_i+2-in-span-i=2}. As in Case 1, we have
	\begin{align*}
		e_{12}\cdot(e_{i,i+1}*e_{i+1,i+2})=e_{12}\cdot(e_{12}*e_{23})=e_{12}*(e_{12}\cdot e_{23})=e_{12}*e_{13}\in\spn_K\{e_{1n}\}
	\end{align*}
	by \cref{e_ik*e_kl-in-spn-e_1n}. Hence, one can discard $(s,t)=(2,n-1)$ from \cref{e_i_i+1*e_i+1_i+2-in-span-i=1}. If $n=3$, we are done. Otherwise, $n>3$, so we have
	\begin{align*}
		(e_{i,i+1}*e_{i+1,i+2})\cdot e_{3n}=(e_{12}*e_{23})\cdot e_{3n}=(e_{12}\cdot e_{23})*e_{3n}
		=e_{13}*e_{3n}\in\spn_K\{e_{1n}\}
	\end{align*}
	 by \cref{e_ik*e_kj-in-span-e_ij-e_1j}, and we can also discard $(s,t)=(2,3)$ from \cref{e_i_i+1*e_i+1_i+2-in-span-i=1}.
\end{proof}

\begin{lem}
	Let $n>4$. Then for all $1\le i<k<j\le n$ we have
	$
		(e_{ik}*e_{kj})(i,j)=(e_{12}*e_{23})(1,3).
	$
\end{lem}
\begin{proof}
	\textit{Case 1.} $i>1$. Then
	\begin{align*}
		(e_{ik}*e_{kj})(i,j)=(e_{1i}\cdot(e_{ik}*e_{kj}))(1,j)=(e_{1i}*(e_{ik}\cdot e_{kj}))(1,j)=(e_{1i}*e_{ij})(1,j).
	\end{align*}
	So, it suffices to consider $i=1$.
	
	\textit{Case 2.1.} $i=1$ and $k>2$. Then 
	\begin{align*}
		(e_{1k}*e_{kj})(1,j)=((e_{1,k-1}\cdot e_{k-1,k})*e_{kj})(1,j)=((e_{1,k-1}*e_{k-1,k})\cdot e_{kj})(1,j)
		=(e_{1,k-1}*e_{k-1,k})(1,k).
	\end{align*}
	If $k-1=2$, we are done. Otherwise, applying the same argument, we get
	\begin{align*}
		(e_{1,k-1}*e_{k-1,k})(1,k)=(e_{1,k-2}*e_{k-2,k-1})(1,k-1),
	\end{align*}
	and so on, until we arrive at $(e_{12}*e_{23})(1,3)$. 
	
	\textit{Case 2.2.} $(i,k)=(1,2)$. If $j<n$, then
	\begin{align*}
		(e_{ik}*e_{kj})(i,j)=(e_{12}*e_{2j})(1,j)=((e_{12}*e_{2j})\cdot e_{jn})(1,n)
		=((e_{12}\cdot e_{2j})*e_{jn})(1,n)=(e_{1j}*e_{jn})(1,n),
	\end{align*}
	and we are in the conditions of the Case 2.1. Let $j=n>4$. Then
	\begin{align*}
		(e_{ik}*e_{kj})(i,j)=(e_{12}*e_{2n})(1,n)=(e_{12}*(e_{23}\cdot e_{3n}))(1,n)
		=(e_{12}\cdot (e_{23}*e_{3n}))(1,n)=(e_{23}*e_{3n})(2,n).
	\end{align*}
	Applying the same argument once again, we get $(e_{23}*e_{3n})(2,n)=(e_{34}*e_{4n})(3,n)$, which coincides with $(e_{13}*e_{3n})(1,n)$ by Case 1. Thus, the result again follows by Case 2.1.
\end{proof}

\begin{thrm}\label{interchangeable-UT_n-descr}
	If $n>4$, then the interchangeable structures on $UT_n(K)$ are exactly the associative products that are linear combinations of the structures:
	\begin{enumerate}
		\item\label{e_i-1i*e_jj+1=e_1n-inter} $\I^1_{i,j} (1<i\le n, 1\le j<n):$ $e_{i-1,i}*e_{j,j+1}=e_{1n}$; 
		
		\item\label{e_12*e_ii+2=e_1n_e_ii+1*e_i+1i+2=e_2n} $\I^2_i (1\le i<n-1):$ 
		$e_{12}*e_{i,i+2}=e_{1n}$, $e_{i,i+1}*e_{i+1,i+2}=e_{2n}$;
		
		\item\label{e_ii+2*e_n-1n=e_1n_e_ii+1*e_i+1i+2=e_1n-1} $\I^3_i (1\le i<n-1):$  $e_{i,i+2}*e_{n-1,n}=e_{1n}$, $e_{i,i+1}*e_{i+1,i+2}=e_{1,n-1}$;
		
		\item\label{e_ik*e_kj=e_ij-inter} $\I^4:$ $e_{ik}*e_{kj}=e_{ij}$ for all $1\le i<k<j\le n$.
	\end{enumerate}
	If $n\in\{3,4\}$, then the interchangeable structures on $UT_n(K)$ are exactly the associative products that are linear combinations of the structures \cref{e_i-1i*e_jj+1=e_1n-inter,e_12*e_ii+2=e_1n_e_ii+1*e_i+1i+2=e_2n,e_ii+2*e_n-1n=e_1n_e_ii+1*e_i+1i+2=e_1n-1} above.
\end{thrm}
\begin{proof}
	Let $*$ be an interchangeable structure on $UT_n(K)$. Then by the results of \cref{sec-e_j*e_kl-inter} there exist $\{\af_{ij}: 1<i\le n,1\le j<n,i\ne j\}\sst K$, $\{\bt_i: 1\le i\le n-2, i\ne 2\}\sst K$ and $\{\gm_i: 1\le i\le n-2, i\ne 2\}\sst K$ such that
	\begin{align*}
		e_{i-1,i}*e_{j,j+1}&=\af_{ij}e_{1n},\ 1<i\le n,\ 1\le j<n,\ i\ne j,\\
		e_{12}*e_{i,i+2}&=\bt_ie_{1n},\ 1\le i\le n-2,\ i\ne 2,\\
		e_{i,i+2}*e_{n-1,n}&=\gm_ie_{1n},\ 1\le i\le n-2,\ i\ne n-3.
	\end{align*}
	Moreover, by the results of \cref{sec-e_ik*e_kj-inter}, there exist $\{\af_{ii}\}_{i=2}^{n-1}\sst K$ and $\eta\in K$ such that
	\begin{align*}
		e_{ik}*e_{kj}&=\eta e_{ij},\ j>i+2,\ (i,k,j)\not\in\{(1,2,4),(n-3,n-1,n)\},\\
		e_{12}*e_{24}&=\begin{cases}
			\bt_2e_{14}, & n=4,\\
			\bt_2e_{1n}+\eta e_{14}, & n>4,
		\end{cases}\\
		e_{n-3,n-1}*e_{n-1,n}&=\begin{cases}
			\gm_1e_{14}, & n=4,\\
			\gm_{n-3}e_{1n}+\eta e_{n-3,n}, & n>4,
		\end{cases}\\
		e_{i,i+1}*e_{i+1,i+2}&=\begin{cases}
			\gm_ie_{12}+\af_{i+1,i+1}e_{13}+\bt_ie_{23}, & n=3,\\
			\gm_ie_{13}+\af_{i+1,i+1}e_{14}+\bt_ie_{24}, & n=4,\\
			\eta e_{i,i+2}+\gm_ie_{1,n-1}+\af_{i+1,i+1}e_{1n}+\bt_ie_{2n}, & n>5,
		\end{cases}\\
		&\quad\quad 1\le i\le n-2.
	\end{align*}
	Thus, $*$ is the linear combination $\sum_{i,j} \af_{ij}\I^1_{i,j}+\sum_i\bt_i\I^2_i+\sum_i\gm_i\I^3_i+\eta\I^4$. Observe that  $\I^4=\I^1_{2,2}$, whenever $n=3$, and $\I^4=\I^2_2+\I^3_2$, whenever $n=4$.
	
	Conversely, each of the products \cref{e_i-1i*e_jj+1=e_1n-inter,e_12*e_ii+2=e_1n_e_ii+1*e_i+1i+2=e_2n,e_ii+2*e_n-1n=e_1n_e_ii+1*e_i+1i+2=e_1n-1} is interchangeable with $\cdot$ as seen in \cref{sec-interchangeable-classes}. Moreover, \cref{e_ik*e_kj=e_ij-inter} coincides with the original product $\cdot$ on $UT_n(K)$, so it is interchangeable with $\cdot$. Thus, any associative linear combination of \cref{e_i-1i*e_jj+1=e_1n-inter,e_12*e_ii+2=e_1n_e_ii+1*e_i+1i+2=e_2n,e_ii+2*e_n-1n=e_1n_e_ii+1*e_i+1i+2=e_1n-1,e_ik*e_kj=e_ij-inter} is an interchangeable structure on $UT_n(K)$.
\end{proof}

\begin{rem}
	There are $(n-1)^2$ products of the form \cref{e_i-1i*e_jj+1=e_1n-inter}, $n-2$ products of the form \cref{e_12*e_ii+2=e_1n_e_ii+1*e_i+1i+2=e_2n} and $n-2$ products of the form \cref{e_ii+2*e_n-1n=e_1n_e_ii+1*e_i+1i+2=e_1n-1}. In total, there are $n^2-3$ linearly independent products of the form \cref{e_i-1i*e_jj+1=e_1n-inter,e_12*e_ii+2=e_1n_e_ii+1*e_i+1i+2=e_2n,e_ii+2*e_n-1n=e_1n_e_ii+1*e_i+1i+2=e_1n-1} and $1$ product of the form \cref{e_ik*e_kj=e_ij-inter} (which is linearly independent with \cref{e_i-1i*e_jj+1=e_1n-inter,e_12*e_ii+2=e_1n_e_ii+1*e_i+1i+2=e_2n,e_ii+2*e_n-1n=e_1n_e_ii+1*e_i+1i+2=e_1n-1} only for $n>4$).
\end{rem}

\subsection{The complete classification of interchangeable structures on $UT_3(K)$}\label{sec-inter-n=3}

In order to classify the interchangeable structures on $UT_3(K)$ we first specify \cref{interchangeable-UT_n-descr} to the case $n=3$.
\begin{cor}\label{structures-I_ij-n=3}
	The interchangeable structures on $UT_3(K)$ are exactly the associative products that are linear combinations of the structures:
	\begin{multicols}{2}
		\begin{enumerate}
			\item\label{I^1_21-n=3} $\I^1_{2,1}:$ $e_{12}*e_{12}=e_{13}$;
			\item\label{I^_22-n=3} $\I^1_{2,2}:$ $e_{12}*e_{23}=e_{13}$;
			\item\label{I^_31-n=3} $\I^1_{3,1}:$ $e_{23}*e_{12}=e_{13}$;
			\item\label{I^_32-n=3} $\I^1_{3,2}:$ $e_{23}*e_{23}=e_{13}$;
			\item\label{I^2_1-n=3} $\I^2_1:$ 
			$e_{12}*e_{13}=e_{13}$, $e_{12}*e_{23}=e_{23}$;
			\item\label{I^3_1-n=3} $\I^3_1:$  $e_{13}*e_{23}=e_{13}$, $e_{12}*e_{23}=e_{12}$.
		\end{enumerate}
	\end{multicols}
	i.e. the associative products of the form
	\begin{align}\label{interchangeable-UT_3(K)-general-form}
		\begin{array}{rclrclrcl}
			e_{12}*e_{12}&=&\lb_1 e_{13},& e_{12}*e_{23}&=&\lb_6 e_{12}+\lb_2 e_{13}+\lb_5 e_{23}, & e_{23}*e_{12}&=&\lb_3 e_{13},\\
			e_{12}*e_{13}&=&\lb_5 e_{13},& e_{13}*e_{23}&=&\lb_6 e_{13},
			 & e_{23}*e_{23}&=&\lb_4 e_{13}
		\end{array}
	\end{align}
	for some $\{\lb_i\}_{i=1}^6\sst K$.
\end{cor}

\begin{lem}\label{associators-inter-UT_3}
	We have the following non-trivial associators of the basis elements $\{e_{12},e_{23},e_{13}\}$ of $UT_3(K)$ under \cref{interchangeable-UT_3(K)-general-form}:
	\begin{align*}
		\begin{array}{rclrcl}
			{[}e_{12},e_{12},e_{12}{]}&=&-\lb_1\lb_5e_{13},& {[}e_{12},e_{23},e_{13}{]}&=&\lb_5\lb_6e_{13},\\
			{[}e_{12},e_{12},e_{23}{]}&=&-\lb_5\lb_6e_{12}-\lb_5^2e_{23}-2\lb_2\lb_5e_{13}, & {[}e_{23},e_{12},e_{23}{]}&=&-\lb_4\lb_5e_{13},\\
			{[}e_{12},e_{12},e_{13}{]}&=&-\lb_5^2e_{13}, & {[}e_{23},e_{23},e_{23}{]}&=&\lb_4\lb_6e_{13},\\
			{[}e_{12},e_{23},e_{12}{]}&=&\lb_1\lb_6e_{13}, & {[}e_{13},e_{12},e_{23}{]}&=&-\lb_5\lb_6e_{13},\\
			{[}e_{12},e_{23},e_{23}{]}&=&\lb_6^2e_{12}+\lb_5\lb_6e_{23}+2\lb_2\lb_6e_{13}, & {[}e_{13},e_{23},e_{23}{]}&=&\lb_6^2e_{13}.
		\end{array}
	\end{align*}
	In particular, \cref{interchangeable-UT_3(K)-general-form} is associative if and only if $\lb_5=\lb_6=0$.
\end{lem}

\begin{prop}\label{interchangeable-UT_3-full-descr}
	Any interchangeable structure on $UT_3(K)$ is isomorphic to exactly one of the following structures:
	\begin{enumerate}
		\item\label{star-for-lb_5=0-lb_6=0-lb_1-ne-0-inter} $e_{12}\star e_{12} = \af e_{13},\ e_{12}\star e_{23} = \bt e_{13},\ e_{23}\star e_{12} = \gm e_{13},\ e_{23}\star e_{23} = e_{13}$;
		\item\label{star-for-lb_5=0-lb_6=0-lb_1=0-lb_4-ne-0-inter} $e_{12}\star e_{12} = e_{13},\ e_{12} \star e_{23} = \af e_{13},\ e_{23} \star e_{12} = \bt e_{13}$;
		\item\label{star-for-lb_5=0-lb_6=0-lb_1=0-lb_4=0-inter} $e_{12}\star e_{23}=\af e_{13}$, $e_{23}\star e_{12}=\bt e_{13}$.
	\end{enumerate}
	In particular, any such structure is totally compatible with $\cdot$.
\end{prop}
\begin{proof}
	Let $*$ be an interchangeable structure on $UT_3(K)$. By \cref{structures-I_ij-n=3} the product $*$ is of the form \cref{interchangeable-UT_3(K)-general-form}, where $\lb_5=\lb_6=0$ thanks to \cref{associators-inter-UT_3}. So, $*$ is exactly the structure \cref{(12)-matching-UT_3(K)-general-form} with $\lb_5=\lb_6=0$. Such structures were classified in Case 3 of \cref{(12)-matching-UT_3-full-descr}, they are given in \cref{star-for-lb_5=0-lb_6=0-lb_4-ne-0-(12),star-for-lb_5=0-lb_6=0-lb_4=0-lb_1-ne-0-(12),star-for-lb_5=0-lb_6=0-lb_1=0-lb_4=0-(12)} of \cref{(12)-matching-UT_3-full-descr}. So, we just copy them here.
\end{proof}

\begin{cor}
	The totally compatible structures on $UT_3(K)$ are exactly \cref{star-for-lb_5=0-lb_6=0-lb_1-ne-0-inter,star-for-lb_5=0-lb_6=0-lb_1=0-lb_4-ne-0-inter,star-for-lb_5=0-lb_6=0-lb_1=0-lb_4=0-inter} of \cref{interchangeable-UT_3-full-descr}.\footnote{This can also be seen as a particular case of \cref{totally-comp-UT_n-descr} below.}
\end{cor}

\section{Totally compatible structures on $UT_n(K)$}\label{sec-totally-comp}

We are finally ready to characterize totally compatible structures on $UT_n(K)$. 
\begin{thrm}\label{totally-comp-UT_n-descr}
	If $n>3$, then the totally compatible structures on $UT_n(K)$ are exactly the linear combinations of the structures:
	\begin{enumerate}
		\item\label{e_i-1i*e_jj+1=e_1n-totally-comp} $\T^1_{i,j} (1<i\le n, 1\le j<n):$ $e_{i-1,i}*e_{j,j+1}=e_{1n}$; 
		
		\item\label{e_ik*e_kj=e_ij-totally-comp} $\T^2:$ $e_{ik}*e_{kj}=e_{ij}$ for all $1\le i<k<j\le n$.
	\end{enumerate}
	If $n=3$, then the totally compatible structures on $UT_n(K)$ are exactly the linear combinations of the structures \cref{e_i-1i*e_jj+1=e_1n-totally-comp} above.
\end{thrm}
\begin{proof}
	Let $*$ be a product on $UT_n(K)$ that is totally compatible with $\cdot$. In particular, $*$ is interchangeable with $\cdot$, so $*$ is a linear combination
	\begin{align}\label{sum_ij-af_ijI^1_ij+sum_i-bt_iI^2_i+sum_i-gm_iI^3_i+dlI^4}
		\sum_{i,j}\af_{ij}\I^1_{i,j}+\sum_{i=1}^{n-2} \bt_i\I^2_i+\sum_{i=1}^{n-2}\gm_i\I^3_i+\eta\I^4
	\end{align}
	by \cref{interchangeable-UT_n-descr}. However, $\I^1_{i,j}$ is totally compatible with $\cdot$ for all $i,j$ as a structure of the form \cref{annihilator-valued-structure} and $\I^4$ coincides with $\cdot$, so it is also totally compatible with $\cdot$. So, subtracting $\sum_{i,j}\af_{ij}\I^1_{i,j}+\eta\I^4$, one can assume that $\af_{ij}=\eta=0$ in \cref{sum_ij-af_ijI^1_ij+sum_i-bt_iI^2_i+sum_i-gm_iI^3_i+dlI^4} (by doing so one can a priori lose the associativity, but \cref{totally-comp-ass-identity} will be preserved). Thus, it is enough to consider $*$ of the form
	\begin{align*}
		e_{12}*e_{i,i+2}=\bt_ie_{1n},\ e_{i,i+1}*e_{i+1,i+2}=\bt_ie_{2n}+\gm_ie_{1,n-1},\ e_{i,i+2}*e_{n-1,n}=\gm_ie_{1n},\ 1\le i\le n-2.
	\end{align*}
	Given such a $*$, we have
	\begin{align*}
		e_{12}\cdot(e_{i,i+1}*e_{i+1,i+2})=e_{12}\cdot(\bt_ie_{2n}+\gm_ie_{1,n-1})=\bt_ie_{1n}
	\end{align*}
	for all $i$. On the other hand, if $n\ne 4$, then $(e_{12}\cdot e_{i,i+1})*e_{i+1,i+2}=0$ yielding $\bt_i=0$ for all $i$; and if $n=4$, then $(e_{12}\cdot e_{12})*e_{23}=0$ and $(e_{12}\cdot e_{23})*e_{34}=\gm_1e_{14}$ yielding $\bt_1=0$ and $\bt_2=\gm_1$. Similarly,
	\begin{align*}
		(e_{i,i+1}*e_{i+1,i+2})\cdot e_{n-1,n}=(\bt_ie_{2n}+\gm_ie_{1,n-1})\cdot e_{n-1,n}=\gm_ie_{1n}
	\end{align*}
	for all $i$. On the other hand, if $n\ne 4$, then $e_{i,i+1}*(e_{i+1,i+2}\cdot e_{n-1,n})=0$ yielding $\gm_i=0$ for all $i$; and if $n=4$, then $e_{12}*(e_{23}\cdot e_{n-1,n})=\bt_2e_{1n}$ and $e_{23}*(e_{34}\cdot e_{n-1,n})=0$ yielding $\gm_1=\bt_2$ and $\gm_2=0$. Thus, $*$ is zero for $n\ne 4$. For $n=4$ we have $\bt_1=\gm_2=0$ and $\bt_2=\gm_1=:\af$, in which case $*$ is of the form
	\begin{align*}
		e_{12}*e_{24}=\af e_{14},\ e_{12}*e_{23}=\af e_{13},\ e_{23}*e_{34}=\af e_{24},\ e_{13}*e_{34}=\af e_{14},
	\end{align*}
	i.e., $*$ is a multiple of the original product $\cdot$.
	
	Observe that in the case $n=3$ we have $\T^2=\T^1_{2,2}$, so the linear combinations of $\T^1_{i,j}$ already give all the totally compatible structures on $UT_n(K)$.
\end{proof}

\begin{rem}
	There are $(n-1)^2$ structures of the form \cref{e_i-1i*e_jj+1=e_1n-totally-comp}. Thus, the number of linearly independent totally compatible structures on $UT_n(K)$ is either $(n-1)^2+1$ (if $n>3$) or $(n-1)^2$ (if $n=3$).
\end{rem}

	\section*{Acknowledgements}
	
	The author thanks the referee for pointing out typos and for the suggestions that helped to improve the clarity of the paper.
	
	\section*{Declarations}
	
	%\subsection*{Funding statement}
	%The author was partially supported by CMUP, member of LASI, which is financed by national funds through FCT --- Fundação para a Ciência e a Tecnologia, I.P., under the project with reference UIDB/00144/2020.
	
	\subsection*{Data availability statement}
	Data sharing is not applicable to this article as no datasets were generated or analysed during the current study.
	
	\subsection*{Conflict of interest statement}
	Declarations of interest: none.

	\bibliography{bibl}{}
	\bibliographystyle{acm}
	
\end{document}